\documentclass[appendixprefix=Appendix, 11pt]{article}
\usepackage[utf8]{inputenc}

\usepackage[margin=1in]{geometry}
\usepackage{amsmath, amssymb, bm, xcolor, graphics}
\usepackage{algorithm, algorithmicx, algpseudocode}
\usepackage[color=green!40]{todonotes}
\usepackage[style=numeric, sorting=nyt, maxbibnames=99]{biblatex}
\usepackage{lipsum}
\usepackage{caption}
\usepackage{subcaption}
\usepackage{dsfont}
\usepackage[colorlinks=true,
            breaklinks=true,
            bookmarks=true,
            urlcolor=blue,
            citecolor=blue,
            linkcolor=blue,
            bookmarksopen=false,
            draft=false]{hyperref}
\def\EMAIL#1{\href{mailto:#1}{#1}} 

\usepackage{amsthm}
\usepackage{authblk}
\usepackage{lmodern, setspace}
\usepackage[page]{appendix}

\DeclareMathOperator*{\argmin}{argmin}

\newtheorem{theorem}{Theorem}
\newtheorem{lemma}{Lemma}

\newtheorem{assumption}{Assumption}
\newtheorem{remark}{Remark}

\title{LP-based Approximations for Disjoint Bilinear and Two-Stage Adjustable Robust Optimization}
\author{
  \vspace{3mm}
  Omar El Housni\\
  \vspace{-4mm}
  {\footnotesize Operations Research and Information Engineering, Cornell University, New York, USA, 
  \EMAIL{oe46@cornell.edu}}
  \and
  \vspace{-3mm}
  Ayoub Foussoul\\
  \vspace{-4mm}
  {\footnotesize Industrial Engineering and Operations Research, Columbia University, New York, USA, 
  \EMAIL{af3209@columbia.edu}}
  \and
  \vspace{-3mm}
  Vineet Goyal\\
  \vspace{-4mm}
  {\footnotesize Industrial Engineering and Operations Research, Columbia University, New York, USA, 
  \EMAIL{vg2277@columbia.edu}}
}
\date{}

\begin{document}

\maketitle

\begin{abstract} 
We consider the class of disjoint bilinear programs $ \max \, \{ \mathbf{x}^T\mathbf{y} \mid \mathbf{x} \in \mathcal{X}, \;\mathbf{y} \in \mathcal{Y}\}$ where $\mathcal{X}$ and $\mathcal{Y}$ are packing polytopes. We present an $O(\frac{\log \log m_1}{\log m_1}  \frac{\log \log m_2}{\log m_2})$-approximation algorithm for this problem where $m_1$ and $m_2$ are the number of packing constraints in $\mathcal{X}$ and $\mathcal{Y}$ respectively. In particular, we show that there exists a near-optimal solution $(\Tilde{\mathbf{x}}, \Tilde{\mathbf{y}})$ such that $\Tilde{\mathbf{x}}$ and $\Tilde{\mathbf{y}}$ are ``near-integral". We give an LP relaxation of the problem from which we obtain the near-optimal near-integral solution via randomized rounding. We show that our relaxation is tightly related to the widely used reformulation linearization technique (RLT). As an application of our techniques, we present a tight approximation for the two-stage adjustable robust optimization problem with covering constraints and right-hand side uncertainty where the separation problem is a bilinear optimization problem. In particular, based on the ideas above, we give an LP restriction of the two-stage problem that is an $O(\frac{\log n}{\log \log n} \frac{\log L}{\log \log L})$-approximation where $L$ is the number of constraints in the uncertainty set. This significantly improves over state-of-the-art approximation bounds known for this problem. Furthermore, we show that our LP restriction gives a feasible affine policy for the two-stage robust problem with the same (or better) objective value. As a consequence, affine policies give an  $O(\frac{\log n}{\log \log n} \frac{\log L}{\log \log L})$-approximation of the two-stage problem, significantly generalizing the previously known bounds on their performance.
\vspace{2mm}


\end{abstract}

\section{Introduction}

We consider the following class of disjoint bilinear programs,
\begin{gather}
\label{eq:pdb}
\tag{PDB}
\begin{gathered}
    z_{\sf{PDB}} = \max_{\mathbf{x}, \mathbf y} \; \{\mathbf{x}^T\mathbf{y} \mid \mathbf{x} \in \mathcal{X}, \;\mathbf{y} \in \mathcal{Y}\},
\end{gathered}
\end{gather}
where $\mathcal{X}$ and $\mathcal{Y}$ are packing polytopes given by an intersection of knapsack constraints. Specifically, $$\mathcal{X} := \{\mathbf{x} \geq \mathbf{0} \mid \mathbf{P}\mathbf{x} \leq \mathbf{p}\}$$ and
$$\mathcal{Y} := \{\mathbf{y} \geq \mathbf{0} \mid \mathbf{Q}\mathbf{y} \leq \mathbf{q}\} ,$$ where $\mathbf{P} \in \mathbb{R}_+^{m_1 \times n}$, $\mathbf{Q} \in \mathbb{R}_+^{m_2 \times n}$, $\mathbf{p} \in \mathbb{R}_+^{m_1}$ and $\mathbf{q} \in \mathbb{R}_+^{m_2}$. We refer to this problem as a packing disjoint bilinear program \ref{eq:pdb}. This is a subclass of the well-studied disjoint bilinear problem: $\underset{\mathbf{x}, \mathbf{y}}{\max} \; \{\mathbf{x}^T    \mathbf{M} \mathbf{y}\; | \; \mathbf{x} \in \mathcal{X},\; \mathbf{y} \in \mathcal{Y}\} $, where $\mathbf{M}$ is a general $n \times n$ matrix. 

Disjoint bilinear programming is NP-hard in general (Chen et al. \cite{chen2007complexity}). We show that it is NP-hard to even approximate within any finite factor. Several heuristics have been studied for this problem including cutting-planes algorithms (Konno et al. \cite{Konno1976cuttingplane}), polytope generation methods (Vaish et al. \cite{vaishexpandingpolytopes}), Benders decomposition (Geoffrion \cite{GeoffrionGenbenders72}), reduction to concave minimization (Thieu \cite{Thieu88concave}), reformulation linearization techniques (Sherali and Alameddine \cite{sherali1992new}, Adams and Sherali \cite{adams1986tight}, Audet et al. \cite{audet2000branch}, Tawarmalani et al. \cite{tawarmalani2010strong}), mixed integer programming (Gupte et al. \cite{gupte2013solving}, Freire et al. \cite{freire2012integer}) and two-stage robust optimization (Zhen et al. \cite{robustZhen2018}). However, to the best of our knowledge, no approximation algorithms with provable guarantees are known for this problem.

Many important applications can be formulated as a disjoint bilinear program including fixed charge network flows (Rebennack et al. \cite{Rebennack2009fixedchargenetw}), concave cost facility location (Soland \cite{Soland1974FLwithcocavecosts}), bilinear assignment problems (Ćustić et al. \cite{custic2016bilinear_assignement}), non-convex cutting-stock problems (Harjunkoski et al. \cite{Harjunkoski1998trimloss}), multicommodity flow network interdiction problems (Lim and Smith \cite{Lim2007multicomflow}), bimatrix games (Mangasarian and Stone \cite{MANGASARIAN1964twopersongame}, Firouzbakht et al. \cite{Firouzbakht2016bimatrixgameswireless}) pooling problems (Gupte et al. \cite{Gupte2017pooling}).

One important application closely related to disjoint bilinear optimization that we focus on in this paper, is two-stage adjustable robust optimization. In particular, the separation problem of a two-stage adjustable robust problem can be formulated as a disjoint bilinear optimization problem. More specifically, we consider the following two-stage adjustable robust problem,
\begin{gather}
\tag{AR}
\label{eq:ar}
\begin{aligned}
z_{\sf{AR}} = \min_{\mathbf{x}, t} \quad & \mathbf{c}^T\mathbf{x} + t\\
& t \geq \mathcal{Q}(\mathbf{x}),\\
& \mathbf{x} \in \mathcal{X},
\end{aligned}
\end{gather}
where for all $\mathbf{x} \in \mathcal{X}$,
\begin{gather}
\label{eq:qx}
\tag*{ }
\begin{aligned}
\mathcal{Q}(\mathbf{x}) = \max_{\mathbf{h} \in \mathcal{U}} \min_{\mathbf{y}\geq \mathbf{0}}\; \{\mathbf{d}^T\mathbf{y} \mid \mathbf{A}\mathbf{x} + \mathbf{B}\mathbf{y} \geq \mathbf{h}\}.
\end{aligned}
\end{gather}
Here $\mathbf{A} \in \mathbb{R}^{m \times n^\prime}$, $\mathbf{B} \in \mathbb{R}_+^{m \times n}$, $\mathbf{c} \in \mathbb{R}_+^{n^\prime}$, $\mathbf{d} \in \mathbb{R}_+^{n}$, $\mathcal{X} \subset \mathbb{R}_+^{n^\prime}$ is a polyhedral cone, and $\mathcal{U}$ is a polyhedral uncertainty set. The separation problem of \ref{eq:ar} is the following: given a candidate solution $(\mathbf{x}, t)$, decide if it is feasible, i.e., $\mathbf{x} \in \mathcal{X}$ and $t \geq \mathcal{Q}(\mathbf{x})$ or give a separating hyperplane. This is equivalent to solving $\hyperref[eq:qx]{\mathcal{Q}(\mathbf{x})}$. We will henceforth refer to $\hyperref[eq:qx]{\mathcal{Q}(\mathbf{x})}$ as the separation problem. For ease of notation, we use $\mathcal{Q}(\mathbf{x})$ to refer to both the problem and its optimal value. In this two-stage problem, the adversary observes the first-stage decision $\mathbf{x}$ and reveals the worst-case scenario of $\mathbf{h} \in \mathcal{U}$. Then, the decision maker selects a second-stage recourse decision $\mathbf{y}$ such that $\mathbf{B}\mathbf{y}$ covers $\mathbf{h} - \mathbf{A}\mathbf{x}$. The goal is to select a first-stage decision such that the total cost in the worst-case is minimized. This model has been widely considered in the literature (Dhamdhere et al. \cite{Dhamdhere2005}, Feige et al. \cite{feige2007RCES}, Gupta et al. \cite{Gupta2009}, Bertsimas and Goyal \cite{Bertsimas2012OnTP}, Bertsimas and Bidkhori \cite{bertsimas2015geometric}, Bertsimas and de Ruiter \cite{Bertsimas2016duality}, Xu et al. \cite{xu2018copositive}, Zhen et al. \cite{zhen2018}, El Housni and Goyal \cite{GoyalElHoussni2017PAP}, El Housni et al. \cite{housni2020facilitylocation,matchingdrivers}), and has many applications including set cover, capacity planning and network design problems under uncertain demand. 

Several uncertainty sets have been considered in the literature including polyhedral uncertainty sets, ellipsoids and norm balls (see Bertsimas et al. \cite{bertsimas2010survey}). Some of the most important uncertainty sets are budget of uncertainty sets (Bertsimas and Sim \cite{Bertsimas2004priceofrobustness}, Gupta et al. \cite{gupta2011knapsack}, El Housni and Goyal \cite{housni2021affine}) and intersections of budget of uncertainty sets such as CLT sets (see Bandi and Bertsimas \cite{Bertsimas2012CLT}) and inclusion-constrained budgeted sets (see Gounaris et al. \cite{2014inclusionConstrainedBudgets}).  These have been widely used in practice. Following this motivation, we consider in this paper the following uncertainty set
$$\mathcal{U} := \{\mathbf{h} \geq \mathbf{0} \mid \mathbf{R} \mathbf{h} \leq \mathbf{r}\},$$ where $\mathbf{R} \in \mathbb{R}_+^{L \times m}$ and $\mathbf{r} \in \mathbb{R}_+^L$. This is a generalization of the previously mentioned sets. We refer to this as a packing uncertainty set.

Feige et al. \cite{feige2007RCES} show that \ref{eq:ar} is NP-hard to approximate within any factor better than $\Omega(\frac{\log n}{\log \log n})$ even in the special case of a single budget of uncertainty set. Bertsimas and Goyal \cite{Bertsimas2012OnTP} give an $O(\sqrt{m})$-approximation in the case where the first-stage matrix $\mathbf{A}$ is non-negative. Recently, El Housni and Goyal \cite{housni2021affine} give an $O(\frac{\log n}{\log \log n})$-approximation in the case of a single budget of uncertainty set and an $O(\frac{\log^2 n}{\log \log n})$-approximation in the case of an intersection of disjoint budgeted sets. In general, they show an $O(\frac{L \log n}{\log \log n})$-approximation in the case of a packing uncertainty set with $L$ constraints. However, this bound scales linearly with $L$. The two-stage robust covering problem was also considered in the discrete case where the variables of the problem are restricted to be in $\{0,1\}^m$. For this problem, Feige et al. \cite{feige2007RCES} and Gupta et al. \cite{Gupta2009} give an $O(\log n \log m)$-approximation and an $O(\log n + \log m)$-approximation respectively in the case where $\mathbf{A}=\mathbf{B} \in \{0,1\}^{m \times n}$ and the uncertainty set $\mathcal{U}$ is given by a cardinality uncertainty set of the form $\mathcal{U} = \{\mathbf{h} \in \{0,1\}^m \mid \sum_{i=1}^m h_i \leq k\}$. Gupta et al. \cite{gupta2011knapsack} consider a more general uncertainty set, namely, intersection of \textit{p-system} and \textit{q-knapsack} and give an $O(pq\log n)$-approximation of the two-stage problem.

The goal of this paper is to provide LP-based approximation algorithms with provable guarantees for the packing disjoint bilinear program as well as the two-stage adjustable robust problem that improve over the approximation bounds known for these problems.

\subsection{Our Contributions}

\subsubsection{A Polylogarithmic Approximation Algorithm for PDB.}

\vspace{2mm}
\noindent \textbf{Algorithm.}
We present an LP based randomized approximation algorithm for \ref{eq:pdb}. Our algorithm relies on a new idea that might be of independent interest. In particular, we show the existence of a near-optimal near-integral solution of this problem. That is, a near-optimal solution $(\Tilde{\mathbf{x}}, \Tilde{\mathbf{y}})$ such that $\Tilde{x}_i \in \{0, \underset{\mathbf{x} \in \mathcal{X}}{\max} \; x_i / \zeta_1 \}$ and $\Tilde{y}_i \in \{0, \underset{\mathbf{y} \in \mathcal{Y}}{\max} \; y_i / \zeta_2\}$ for some small factors $\zeta_1$ and $\zeta_2$. We give an LP relaxation of \ref{eq:pdb}, i.e., a linear program whose optimal cost is greater than the optimum of \ref{eq:pdb}, from which we obtain such $(\Tilde{\mathbf{x}}, \Tilde{\mathbf{y}})$ via randomized rounding. More specifically, we show the following theorem,
\begin{theorem}
\label{th:pdb}
There exists an LP rounding based randomized algorithm  that gives an $O(\frac{\log \log m_1}{\log m_1}\frac{\log \log m_2}{\log m_2})$-approximation to  \ref{eq:pdb}.
\end{theorem}
\vspace{2mm}
\noindent \textbf{Relation to the Reformulation Linearization Technique (RLT).} We show that our LP relaxation of \ref{eq:pdb} is closely related to the reformulation linearization technique (RLT). RLT provides an efficient approximation for non-convex continuous and mixed-integer optimization problems. It was first introduced by Sherali and Adams \cite{adams1986tight,adams1990linearization,adams1993mixed} in the context of binary bilinear problems and has been since then applied to different other problems including general bilinear problems (Sherali and Alameddine \cite{sherali1992new}), mixed-integer linear problems (Sherali and Adams \cite{sherali1994hierarchy,sherali2009reformulation}) and polynomial problems (Sherali and Tuncbilek \cite{sherali1992global}).
We show the existence of a reformulation linearization of \ref{eq:pdb} that is equivalent to our LP relaxation. This provides a new perspective on our LP relaxation and implies a polylogarithmic approximation guarantee on the performance of tighter relaxations of \ref{eq:pdb} such as the well studied relaxations of the RLT hierarchy (Sherali and Adams \cite{sherali1990hierarchy,sherali1994hierarchy}).

\vspace{2mm}
\noindent \textbf{Numerical Experiments.} Our randomized rounding based algorithm gives an approximate solution of \ref{eq:pdb} that is guaranteed to be within $O(\frac{\log m_1}{\log \log m_1} \frac{\log m_2}{\log \log m_2})$ of the optimum with high probability. We study the empirical performance of our solution by comparing the performance of our algorithm with several benchmarks on randomly generated instances. More specifically, we compare our algorithm with the first level relaxation of the RLT hierarchy, which is a widely used LP approximation for bilinear programs and that has been observed to be a good empirical approximation (Sherali et al. \cite{sherali2000reduced}, Sherali and Adams \cite{adams1986tight,adams1993mixed}, Sherali and Alameddine \cite{sherali1992new}). This relaxation gives an approximate solution of \ref{eq:pdb} that we compare with our solution in terms of objective value and running time needed to compute each solution. We also compare our algorithm to the bilinear solver of Gurobi v9.1.2. In particular, we compare the objective value of our solution to the optimal objective computed using the bilinear solver of Gurobi and 
compare the running time of our algorithm with the running time needed by the bilinear solver of Gurobi to compute a solution that is at least as good as our solution. We show that our solution is significantly faster to compute compared to these benchmarks and gives a good approximation of \ref{eq:pdb}.

\subsubsection{A Polylogarithmic Approximation for the Two-Stage Problem \ref{eq:ar}.}
\vspace{2mm}
\noindent \textbf{Algorithm.} We present an LP-based approximation for \ref{eq:ar}. The separation problem for \ref{eq:ar} is a variant of \ref{eq:pdb}. However, the objective is a difference of a bilinear and a linear term making it challenging to approximate. Our approach approximates \ref{eq:ar} directly. In particular, using ideas from our approximation of \ref{eq:pdb}, we give a compact linear restriction of \ref{eq:ar}, that is, a linear program whose optimal objective is greater than the optimum of \ref{eq:ar}, and show that it is a polylogarithmic approximation of \ref{eq:ar}. In particular, we have the following theorem.

\begin{theorem}
\label{th:ar}
There exists an LP restriction of \ref{eq:ar} that gives an $O(\frac{\log n}{\log \log n}\frac{\log L}{\log \log L})$ approximation to \ref{eq:ar}.
\end{theorem}
\noindent Our bound improves significantly over the prior  approximation bound of $O(\frac{L\log n}{\log \log n})$ \cite{housni2021affine} known for this problem. It also shows a striking contrast between the fractional two-stage robust covering problem and its discrete counter part. In fact, the discrete two-stage robust covering problem under \textit{L-knapsack} uncertainty set considered in \cite{gupta2011knapsack} is hard to approximate within any factor better than $L^{\frac{1}{2}-\epsilon}$, for any $\epsilon > 0$. This follows from the hardness of the maximum independent set problem.

\vspace{2mm}
\noindent \textbf{Relation to Affine Policies.} We show that, surprisingly, our LP restriction is tightly related to {\em affine policies}. Affine policies are a widely used approximation technique in dynamic robust optimization. They consist of restricting the second-stage variables $\mathbf{y}$ to be an affine function of the uncertain right-hand side $\mathbf{h}$ (see for example Ben-Tal et al. \cite{bental2004}). It is known that the optimal affine policy can be computed in polynomial time (Ben-Tal et al. \cite{bental2004}). Several approximation bounds are known for affine policies. Bertsimas and Goyal \cite{Bertsimas2012OnTP} show that affine policies achieve a bound of $O(\sqrt{m})$ under a general polyhedral uncertainty set. They also show that such policies are optimal in the case of a simplex uncertainty set. Recently, El Housni and Goyal \cite{housni2021affine} show that affine policies achieve the bound of $O(\frac{\log n}{\log \log n})$ in the case of a single budget of uncertainty set. They also show a bound of $O(L\frac{\log n}{\log \log n})$ in the case of a general intersection of $L$ budget of uncertainty sets. In this paper, we prove stronger bounds for affine policies under the general packing uncertainty set with $L$ constraints. In particular, we show that our LP restriction of \ref{eq:ar} gives a feasible affine policy for the two-stage problem with the same (or better) objective value. This implies the following approximation bound for affine policies.
\begin{theorem}
\label{th:AFF}
Affine policies give an $O(\frac{\log n}{\log \log n}\frac{\log L}{\log \log L})$-approximation to \ref{eq:ar}.
\end{theorem}
\noindent Our analysis is constructive and provides a faster algorithm to compute near-optimal affine policies with approximation ratio $O(\frac{\log n}{\log \log n}\frac{\log L}{\log \log L})$.

\vspace{2mm}
\noindent \textbf{Numerical Experiments.}
We compare the performance of our LP restriction of \ref{eq:ar} with several benchmarks on randomly generated instances. First, we compare our restriction with the optimal affine policy, which is a widely used approximation for the two-stage problem. Second, we compare our restriction with a generalization of the algorithm of El Housni and Goyal \cite{housni2021affine} to packing uncertainty sets. This algorithm was shown to have a good empirical performance for the case of a single budget of uncertainty set. 
We also compare our LP restriction with the lower-bound of Hadjiyiannis et al. \cite{hadjiyiannis2011scenario} who show the bound provides a good empirical approximation of the optimum of \ref{eq:ar}. We show that our restriction is significantly faster to compute compared to these benchmarks and gives a good approximation of \ref{eq:ar}.


\section{A Polylogarithmic Approximation for PDB}
\label{section:pdb}

In this section, we present an $O(\frac{\log \log m_1}{\log m_1}\frac{\log \log m_2}{\log m_2})$-approximation  for \ref{eq:pdb} (Theorem~\ref{th:pdb}). To prove this theorem, we show an interesting structural property of 
\ref{eq:pdb}. In particular, we show that there exists a near-optimal solution of \ref{eq:pdb} that is ``near-integral''.  Let us define
for all $i \in [n]$, $$ \theta_i = \max_{\mathbf{x} \in \mathcal{X}} x_i, \quad  \gamma_i = \max_{\mathbf{y} \in \mathcal{Y}} y_i, \quad \zeta^*_1 = \frac{3\log m_1}{\log\log m_1 } + 2 \quad \text{and} \quad \zeta^*_2 = \frac{3\log m_2}{\log\log m_2} + 2.$$ We formally state our  structural property in the following lemma.

\begin{lemma} {\normalfont (\textbf{Structural Property})}.
\label{lm:structural-property-pdb}
There exists a feasible solution $(\Tilde{\mathbf{x}}, \Tilde{\mathbf{y}})$ of \ref{eq:pdb} whose objective value is within $O(\frac{\log \log m_1}{\log m_1}\frac{\log \log m_2}{\log m_2})$ of the optimum and such that $\Tilde{x}_i \in \{0, \frac{\theta_i}{\zeta_1}\}$ and $\Tilde{y}_i \in \{0, \frac{\gamma_i}{\zeta_2}\}$ for all $i \in [n]$, where $1 \leq \zeta_1 \leq \zeta^*_1$ and  $1 \leq \zeta_2 \leq \zeta^*_2$.
\end{lemma}
\noindent We obtain such a solution satisfying the above property using an LP relaxation of \ref{eq:pdb} via a randomized rounding approach.

\vspace{2mm}
\noindent \textbf{LP Relaxation and Rounding.} We consider the following linear program,

\begin{gather}
\tag{LP-PDB}
\label{eq:lp-pdb}
    z_{\sf{LP-PDB}} = \max_{\bm{\omega}\geq \mathbf{0}} \left\{ \quad \sum_{i=1}^n \theta_i \gamma_i \omega_i \; \left | \;\;
    \begin{aligned}
    & \sum_{i=1}^n \theta_i \mathbf{P}_i \omega_i\leq \mathbf{p} \\
    & \sum_{i=1}^n \gamma_i \mathbf{Q}_i \omega_i \leq \mathbf{q}
    \end{aligned}\;\;\right.\right\},
\end{gather}
where $\mathbf{P}_i$ is the $i$-th column of $\mathbf{P}$ and $\mathbf{Q}_i$ is the $i$-th column of $\mathbf{Q}$. 

We first show that \ref{eq:lp-pdb} is a relaxation of \ref{eq:pdb}.
\begin{lemma}
\label{lm:lppdb-relaxation}
$z_{\sf{PDB}} \leq z_{\sf{LP-PDB}}$.
\end{lemma}
\begin{proof}
Let $(\mathbf{x}^*, \mathbf{y}^*)$ be an optimal solution of \ref{eq:pdb}. 
Let $\bm{\omega}^*$ be such that $\omega^*_i = \frac{x^*_i}{\theta_i}.\frac{y^*_i}{\gamma_i}$ for all $i \in [n]$. By definition, we have $x^*_i \leq \theta_i$ and $y^*_i \leq \gamma_i$ for all $i \in [n]$. Hence, 
\begin{align*}
    \sum_{i=1}^n \theta_i \mathbf{P}_i \omega^*_i 
    &= \sum_{i=1}^n \theta_i \mathbf{P}_i \frac{x^*_i}{\theta_i}\frac{y^*_i}{\gamma_i} \leq \sum_{i=1}^n \mathbf{P}_i x^*_i \leq \mathbf{p},
\end{align*}
and  
\begin{align*}
    \sum_{i=1}^n \gamma_i \mathbf{Q}_i \omega^*_i 
    &= \sum_{i=1}^n \gamma_i \mathbf{Q}_i \frac{x^*_i}{\theta_i}\frac{y^*_i}{\gamma_i} \leq \sum_{i=1}^n \mathbf{Q}_i y^*_i \leq \mathbf{q}.
\end{align*}
Note that we use the fact that $\mathbf{P}$ and $\mathbf{Q}$ are non-negative in the above inequalities. Therefore, $\bm{\omega}^*$ is feasible for \ref{eq:lp-pdb} with objective value $$\sum_{i=1}^{n} \theta_i \gamma_i  \omega^*_i = \sum_{i=1}^{n}  x^*_i y^*_i = z_{\sf{PDB}},$$
which concludes the proof.
\end{proof}
Now, to construct our near-optimal near-integral solution, we consider the randomized rounding approach
described in Algorithm~\ref{alg:rounding-pdb}.
Note that by definition of $\theta_i$, $$\underset{\bm{\omega}}{\max}\{\omega_i \mid \sum_{j=1}^n \theta_j \mathbf{P}_j \omega_j \leq \mathbf{p}, \bm{\omega} \geq \mathbf{0}\} = 1,$$ for all $i \in [n]$. Hence, for all $i \in [n]$, $\omega^*_i$ defined in Algorithm~\ref{alg:rounding-pdb} is such that $\omega^*_i \leq 1$.

\begin{algorithm}
\caption{\,}\label{alg:rounding-pdb}
\hspace*{\algorithmicindent} \textbf{Input:} $\epsilon >0$.\\
\hspace*{\algorithmicindent} \textbf{Output:} a feasible solution verifying the structural property in Lemma~\ref{lm:structural-property-pdb} with probability at least $1-\epsilon-o(1)$.
\begin{algorithmic}[1]
\State Let $\bm{\omega}^*$ be an optimal solution of \ref{eq:lp-pdb} and let $T = 8\lceil{\log \frac{1}{\epsilon}}\rceil$.
\State Initialize $\max = 0$, $\hat{\mathbf{x}} = \mathbf{0}$ and $ \hat{\mathbf{y}} = \mathbf{0}$.
\For{$t=1, \dots, T$}
\State let $\Tilde{\omega}_1, \dots, \Tilde{\omega}_n$ be i.i.d. Bernoulli variables with $\mathbb{P}(\Tilde{\omega}_i = 1) = \omega^*_i$ for  $i \in [n]$.
\State let $\zeta_1^{\min} = \min\{\zeta \geq 1 \;|\; (\theta_1 \Tilde{\omega}_1, \dots, \theta_n \Tilde{\omega}_n)/\zeta \in \mathcal{X}\}$ and $\tilde{\mathbf{x}} = (\theta_1 \Tilde{\omega}_1, \dots, \theta_n \Tilde{\omega}_n)/\zeta_1^{\min}$.
\State let $\zeta_2^{\min} = \min\{\zeta \geq 1 \;|\; (\gamma_1 \Tilde{\omega}_1, \dots, \gamma_n \Tilde{\omega}_n)/\zeta \in \mathcal{Y}\}$ and $\tilde{\mathbf{y}} = (\gamma_1 \Tilde{\omega}_1, \dots, \gamma_n \Tilde{\omega}_n)/\zeta_2^{\min}$.
\If{$\tilde{\mathbf{x}}^T \tilde{\mathbf{y}} \geq \max$ \;}
\State set $\hat{x}_i = \tilde{x}_i$ and $\hat{y}_i = \tilde{y}_i$ for  $i \in [n]$.
\State set $\max = \tilde{\mathbf{x}}^T\tilde{\mathbf{y}}$.
\EndIf
\EndFor
\State \Return $(\hat{\mathbf{x}}, \hat{\mathbf{y}})$
\end{algorithmic}
\end{algorithm}
In our proof of Lemma~\ref{lm:structural-property-pdb}, we use  the following variant of   Chernoff bounds.
\begin{lemma}
\label{lm:chernoff-bounds}
{\normalfont (Chernoff Bounds~\cite{chernoff52}).}\\
\noindent {\normalfont (a)} Let $\chi_1, \dots, \chi_r$ be independent Bernoulli trials. Denote $\Xi := \sum_{i=1}^{r}\epsilon_i \chi_i$ where $\epsilon_1, \dots, \epsilon_r$ are reals in [0,1]. Let $s >0$ such that $\mathbb{E}(\Xi) \leq s$. Then for any $ \delta >0$ we have,
\begin{align*}
    \mathbb{P}(\Xi \geq (1+\delta)s) \leq \left(\frac{e^{\delta}}{(1+\delta)^{1+\delta}}\right)^{s}.
\end{align*}

\vspace{2mm}

\noindent {\normalfont (b)} Let $\chi_1, \dots, \chi_r$ be independent Bernoulli trials. Denote $\Xi := \sum_{i=1}^{r} \epsilon_i \chi_i$ where $\epsilon_1, \dots, \epsilon_r$ are reals in (0,1]. Then for any $0 < \delta < 1$,
\begin{align*}
    \mathbb{P}(\Xi \leq (1-\delta)\mathbb{E}(\Xi)) \leq e^{-\frac{1}{2}\delta^2 \mathbb{E}(\Xi)}.
\end{align*}
\end{lemma}
\noindent For the sake of completeness, we present the proof for these bounds in Appendix \ref{appendix:chernoff-bounds}.

\vspace{2mm} \noindent{\em Proof  of Lemma~\ref{lm:structural-property-pdb}.}
We show that at any iteration of Algorithm~\ref{alg:rounding-pdb}, the vectors $\tilde{\mathbf{x}}, \tilde{\mathbf{y}} $ verify the structural property Lemma \ref{lm:structural-property-pdb} with constant probability. To show this, let us fix an iteration $t$ of the algorithm. It is sufficient to show that the following inequalities hold with constant probability,

\begin{gather}
\label{eq:properties-pdb}
\begin{aligned}
 \sum_{i=1}^n \mathbf{P}_i \frac{\theta_i\Tilde{\omega}_i}{\zeta^*_1} & \leq \mathbf{p}, \\
 \sum_{i=1}^n \mathbf{Q}_i \frac{\gamma_i\Tilde{\omega}_i}{\zeta^*_2} & \leq \mathbf{q}, \\
 \sum_{i=1}^n \frac{\theta_i\Tilde{\omega}_i}{\zeta^*_1}\frac{\gamma_i\Tilde{\omega}_i}{\zeta^*_2}  &\geq \frac{z_{\sf{LP-PDB}}}{2\zeta^*_1 \zeta^*_2}. 
\end{aligned}
\end{gather}
In fact, if the inequalities~\eqref{eq:properties-pdb} hold with constant probability, then with constant probability, the two vectors $(\theta_1\Tilde{\omega}_1, \dots, \theta_1\Tilde{\omega}_n)/\zeta_1^*$ and $(\gamma_1\Tilde{\omega}_1, \dots, \gamma_n\Tilde{\omega}_n)/\zeta_1^*$ are feasible for \ref{eq:pdb} implying that that $1 \leq \zeta_1^{\min} \leq \zeta_1^*$ and $1 \leq \zeta_2^{\min} \leq \zeta_2^*$ in this iteration. Moreover, these vectors have an objective value at least $\frac{1}{2\zeta^*_1 \zeta^*_2}z_{\sf{LP-PDB}}$. which is by Lemma~\ref{lm:lppdb-relaxation} greater than $\frac{1}{2\zeta^*_1 \zeta^*_2}z_{\sf{PDB}}$. This implies that the feasible solution $\Tilde{\mathbf{x}}, \Tilde{\mathbf{y}}$ of \ref{eq:pdb} has an objective value at least $O(\frac{\log \log m_1}{\log m_1}\frac{\log \log m_2}{\log m_2})$ of the optimum.

\noindent First, we have,
\begin{gather}
\tag*{ }
\label{ineq1}
\begin{aligned}
\mathbb{P}\left(\sum_{i=1}^n \theta_i\mathbf{P}_i\frac{\Tilde{\omega}_i}{\zeta^*_1} \nleq \mathbf{p}\right)
& \leq \sum_{j=1}^{m_1} \mathbb{P}\left(\sum_{i=1}^n \theta_i P_{ji}\frac{\Tilde{\omega}_i}{\zeta^*_1} > p_j\right) \\
&= \sum_{j \in [m_1]: p_j > 0} \mathbb{P}\left(\sum_{i=1}^n \frac{\theta_i P_{ji}}{p_j}\Tilde{\omega}_i > \zeta^*_1\right) \\
&\leq \sum_{j \in [m_1]: p_j > 0} \left(\frac{e^{\zeta^*_1-1}}{(\zeta^*_1)^{\zeta^*_1}}\right) \\
&\leq m_1 \frac{e^{\zeta^*_1-1}}{(\zeta^*_1)^{\zeta^*_1}},
\end{aligned}
\end{gather}
where the first inequality  follows from a union bound on $m_1$ constraints. The second equality holds because for all $j \in [m_1]$ such that $p_j = 0$, we have
$$\mathbb{P}\left(\sum_{i=1}^n \theta_i P_{ji}\frac{\Tilde{\omega}_i}{\zeta^*_1} > p_j\right) = 0.$$ 
In fact, if $p_j = 0$, by feasibility of $\bm{\omega}^*$ in \ref{eq:lp-pdb}, we get $$\sum_{i=1}^n \theta_i P_{ji}\frac{\omega^*_i}{\zeta^*_1} = 0,$$  hence, we almost surely have $$\sum_{i=1}^n \theta_i P_{ji}\frac{\Tilde{\omega}_i}{\zeta^*_1} = 0.$$  The second inequality follows from the Chernoff bounds $\hyperref[lm:chernoff-bounds]{(a)}$ with $\delta = \zeta^*_1 - 1$ and $s=1$. In particular, $\frac{\theta_i P_{ji}}{p_j} \in [0,1]$ by definition of $\theta_i$ for all $i \in [n]$ and $j \in [m_1]$ such that $p_j > 0$ and for all $j \in [m_1]$ such that $p_j > 0$ we have, $$\mathbb{E}\left[\sum_{i=1}^n \frac{\theta_i P_{ji}}{p_j}\Tilde{\omega}_i\right] = \sum_{i=1}^n \frac{\theta_i P_{ji}}{p_j}\omega^*_i \leq 1 ,$$ which holds by feasibility of $\bm{\omega}^*$. Next, note that

\begin{align*}
    \frac{e^{\zeta^*_1-1}}{(\zeta^*_1)^{\zeta^*_1}}
    &= e^{\zeta^*_1 - 1 - \zeta^*_1 \log \zeta^*_1}\\
    &= e^{- \zeta^*_1 \log \zeta^*_1 + o(\zeta^*_1 \log \zeta^*_1)}
\end{align*}

\noindent and,

\begin{align*}
    \zeta^*_1 \log \zeta^*_1
    &= \left(\frac{3\log m_1}{\log \log m_1}+2\right) \log \left(\frac{3\log m_1}{\log \log m_1}+2 \right)\\
    &=3\log m_1 + o(\log m_1)
\end{align*}

\noindent Hence, $$\frac{e^{\zeta^*_1-1}}{(\zeta^*_1)^{\zeta^*_1}} = e^{-3\log m_1 + o(\log m_1)}= O(e^{-2\log m_1}) = O(\frac{1}{m_1^2})$$Therefore, there exists a constant $c > 0$ such that,
\begin{gather}
    \label{eq:property1-pdb}
    \mathbb{P}(\sum_{i=1}^n \mathbf{P}_i\frac{\theta_i\Tilde{\omega}_i}{\zeta^*_1} \nleq \mathbf{p})  \leq \frac{c}{m_1}.
\end{gather}
By a similar argument there exists a constant $c' > 0$, such that
\begin{gather}
\label{eq:property2-pdb}
    \begin{aligned}
    \mathbb{P}(\sum_{i=1}^n \mathbf{Q}_i\frac{\gamma_i\Tilde{\omega}_i}{\zeta^*_2} \nleq \mathbf{q}) \leq \frac{c'}{m_2},
\end{aligned}
\end{gather}
Finally we have,
\begin{gather}
\label{eq:property3-pdb}
\begin{aligned}
    \mathbb{P}\left(\sum_{i=1}^n \frac{\theta_i\gamma_i\Tilde{\omega}^2_i}{\zeta^*_1\zeta^*_2} < \frac{1}{2\zeta^*_1\zeta^*_2} \sum_{i=1}^n \theta_i\gamma_i\omega^*_i\right)
    &= \mathbb{P}\left(\sum_{i=1}^n \frac{\theta_i\gamma_i}{\sum_{j=1}^n \theta_j\gamma_j\omega^*_j}\Tilde{\omega}_i < \frac{1}{2}\right) \leq e^{-\frac{1}{8}},
\end{aligned}
\end{gather}
where the inequality follows from Chernoff bounds $\hyperref[lm:chernoff-bounds]{(b)}$ with $\delta=1/2$. In particular,  for  $i \in [n]$, $$\frac{\theta_i\gamma_i}{\sum_{j=1}^n \theta_j\gamma_j\omega^*_j} \leq 1.$$  This is because the unit vector $\mathbf{e}_i$ is feasible for \ref{eq:lp-pdb} for all $i \in [n]$ which implies $$\theta_i\gamma_i \leq z_{\sf{LP-PDB}} = \sum_{j=1}^n \theta_j\gamma_j\omega^*_j,$$ and we also have, $$\mathbb{E}\left[\sum_{i=1}^n \frac{\theta_i\gamma_i}{\sum_{j=1}^n \theta_j\gamma_j\omega^*_j}\Tilde{\omega}_i\right] = 1.$$
Combining inequalities \eqref{eq:property1-pdb}, \eqref{eq:property2-pdb} and \eqref{eq:property3-pdb} we get that properties \eqref{eq:properties-pdb} hold with the probability at least $$1-\frac{c}{m_1}-\frac{c'}{m_2}-e^{-\frac{1}{8}} = 1-e^{-
\frac{1}{8}} - o(1),$$ which is greater than a constant for $m_1$ and $m_2$ large enough. \qed

\vspace{2mm}\noindent {\em Proof  of Theorem~\ref{th:pdb}.} Let $(\hat{\mathbf{x}}, \hat{\mathbf{y}})$ be the output solution of Algorithm~\ref{alg:rounding-pdb}. Then $(\hat{\mathbf{x}}, \hat{\mathbf{y}})$ has an objective value at least $O(\frac{\log \log m_1}{\log m_1}\frac{\log \log m_2}{\log m_2})$ of optimum of \ref{eq:pdb} if and only if $(\Tilde{\mathbf{x}}, \Tilde{\mathbf{y}})$ does too at some iteration of the main loop.
From our proof of Lemma~\ref{lm:structural-property-pdb}, this happens with probability at least
\begin{align*}
    1-(e^{-\frac{1}{8}} - o(1))^T
    &=  1-e^{-\frac{T}{8}} - o(1)\\
    &=  1-e^{-\lceil{\log \frac{1}{\epsilon}}\rceil} - o(1)\\
    &\geq 1 - \epsilon - o(1).
\end{align*}
Therefore, with probability at least $1-\epsilon - o(1)$, Algorithm~\ref{alg:rounding-pdb} outputs a feasible solution of \ref{eq:pdb} whose objective value is within $O(\frac{\log\log m_1}{\log m_1} \frac{\log\log m_2}{\log m_2})$ of $z_{\sf{PDB}}$. \qed

\vspace{3mm}
\noindent\textbf{Hardness of the General Disjoint Bilinear Program.} Like packing linear programs, the covering linear programs are known to have logarithmic integrality gaps. Hence, a natural question to ask would be whether similar results can be proven for an equivalent covering version of \ref{eq:pdb}, i.e., a disjoint bilinear program of the form,
\begin{gather}
\tag{CDB}
\label{eq:cdb}
\begin{aligned}
z_{\sf{cdb}} = \min_{\mathbf{x}, \mathbf{y}} \{\mathbf{x}^T\mathbf{y} \mid \mathbf{P}\mathbf{x} \geq \mathbf{p}, \; \mathbf{Q}\mathbf{y} \geq \mathbf{q},\; \mathbf{x}, \mathbf{y} \geq \mathbf{0}\}
\end{aligned}
\end{gather}
where $\mathbf{P} \in \mathbb{R}_+^{m_1 \times n}$, $\mathbf{Q} \in \mathbb{R}_+^{m_2 \times n}$, $\mathbf{p} \in \mathbb{R}_+^{m_1}$ and $\mathbf{q} \in \mathbb{R}_+^{m_2}$. However, the previous analysis does not extend to the covering case. In particular, we have the following inapproximability result.

\begin{theorem}
\label{th:inapproximability-cdb}
The covering disjoint bilinear program \ref{eq:cdb} is NP-hard to approximate within any finite factor.
\end{theorem}

\noindent The proof of Theorem~\ref{th:inapproximability-cdb} uses a polynomial time transformation from the \textit{Monotone Not-All-Equal 3-Satisfiability} (MNAE3SAT) problem and is given in Appendix \ref{appendix:proof-inapproximability}.

\section{Relation to the Reformulation Linearization Technique (RLT).}
\label{section:rlt}

The reformulation linearization technique (RLT) is a widely used approach to construct tight linear relaxations of non-convex continuous and mixed-integer optimization problems. It was first proposed by Sherali and Adams \cite{adams1986tight,adams1990linearization,adams1993mixed} in the context of binary bilinear programming and has been since then applied to different other problems including general bilinear problems (Sherali and Alameddine \cite{sherali1992new}), mixed-integer linear problems (Sherali and Adams \cite{sherali1994hierarchy,sherali2009reformulation}) and polynomial problems (Sherali and Tuncbilek \cite{sherali1992global}). RLT constructs an LP relaxation in two phases: a reformulation phase where some valid polynomial constraints are added to the problem. Then a linearization phase where the monomial terms of the resulting problem are linearized. The resulting relaxation can then be combined with a branch-and-bound framework to solve the problem to optimality (see Sherali and Alameddine \cite{sherali1992new} for example). We show that our LP relaxation \ref{eq:lp-pdb} is tightly related to RLT. In particular, we show the existence of a reformulation linearization of \ref{eq:pdb} that is equivalent to \ref{eq:lp-pdb}. 


Let us begin by writing \ref{eq:pdb} in the following equivalent epigraph form,

\begin{gather*}
\begin{aligned}
z_{\sf PDB} = \max_{\mathbf{x}, \mathbf{y}, \mathbf{u}} \quad & \sum_{i=1}^n \theta_i \gamma_i u_{ii}\\
& u_{ij} \leq x_i y_j, \quad \forall i, j\\
& \sum_{i=1}^n \theta_i\mathbf{P}_ix_i \leq \mathbf{p}, \quad \quad  \sum_{i=1}^n \gamma_i\mathbf{Q}_iy_i \leq \mathbf{q},\\
& 0 \leq x_i \leq 1, \quad 0 \leq y_i \leq 1, \quad \forall i.\\
\end{aligned}
\end{gather*}

\vspace{2mm}
\noindent \textbf{Reformulation Phase.}
The reformulation phase consists of adding to the above program a set of valid polynomial constraints that one gets from multiplying the linear constraints by terms of the form $\prod_{i \in S} x_i \prod_{i \in S'} y_i \prod_{i \in T} (1-x_i) \prod_{i \in T'} (1-y_i)$ where $S, S', T, T' \subset [n]$. We add the polynomial constraints we get from multiplying the linear constraints involving $\mathbf{x}$ by $y_j$ for all $j \in [n]$ and the constraints involving $\mathbf{y}$ by $x_j$ for all $j \in [n]$. We get the following equivalent formulation of \ref{eq:pdb},
\begin{gather*}
\begin{aligned}
z_{\sf PDB} = \max_{\mathbf{x}, \mathbf{y}, \mathbf{u}} \quad & \sum_{i=1}^n \theta_i \gamma_i u_{ii}\\
& u_{ij} \leq x_i y_j, \quad \forall i, j\\
& \sum_{i=1}^n \theta_i\mathbf{P}_ix_i \leq \mathbf{p}, \quad \quad \sum_{i=1}^n \theta_i\mathbf{P}_ix_iy_j \leq \mathbf{p}y_j, \quad \forall j\\
& \sum_{i=1}^n \gamma_i\mathbf{Q}_iy_i \leq \mathbf{q}, \quad \quad \sum_{i=1}^n \gamma_i\mathbf{Q}_iy_ix_j \leq \mathbf{q}x_j, \quad \forall j\\
& 0 \leq x_i \leq 1, \quad 0 \leq y_i \leq 1, \quad \forall i,\\
& 0 \leq x_iy_j \leq y_j, \quad 0 \leq y_ix_j \leq x_j, \quad \forall i,j.\\
\end{aligned}
\end{gather*}



\vspace{2mm}
\noindent \textbf{Linearization Phase.}
In the linearization phase, the bilinear terms $x_iy_j$ in the above LP are replaced with their lower-bounds $u_{ij}$. We get the following linear relaxation of \ref{eq:pdb},

\begin{gather}
\label{eq:rlt-pdb}
\tag{RLT-PDB}
\begin{aligned}
z_{\sf RLT-PDB} = \max_{\mathbf{x}, \mathbf{y}, \mathbf{u}} \quad & \sum_{i=1}^n \theta_i \gamma_i u_{ii}\\
& \sum_{i=1}^n \theta_i\mathbf{P}_ix_i \leq \mathbf{p},\\
& \sum_{i=1}^n \theta_i\mathbf{P}_iu_{ij} \leq \mathbf{p}y_j, \quad \forall j\\
& \sum_{i=1}^n \gamma_i\mathbf{Q}_iy_i \leq \mathbf{q},\\
& \sum_{i=1}^n \gamma_i\mathbf{Q}_iu_{ji} \leq \mathbf{q}x_j, \quad \forall j\\
& 0 \leq x_i \leq 1, \quad 0 \leq y_i \leq 1, \quad \forall i,\\
& 0 \leq u_{ij} \leq y_j, \quad 0 \leq u_{ji} \leq x_j, \quad \forall i,j.\\
\end{aligned}
\end{gather}
We show that the relaxation \ref{eq:rlt-pdb} above is equivalent to \ref{eq:lp-pdb}. First, take an optimal solution $\mathbf{x}^*, \mathbf{y}^*, \mathbf{u}^*$ of \ref{eq:rlt-pdb}, we have $$\sum_{i=1}^n \theta_i\mathbf{P}_iu^*_{i,i} \leq \sum_{i=1}^n \theta_i\mathbf{P}_ix^*_{i} \leq \mathbf{p}$$ and $$\sum_{i=1}^n \gamma_i\mathbf{Q}_iu^*_{i,i} \leq \sum_{i=1}^n \gamma_i\mathbf{Q}_iy^*_{i} \leq \mathbf{q}.$$ Moreover, $0 \leq u^*_{i,i} \leq 1$ for all $i$. Hence, $\bm{\omega} \in [0,1]^n$ defined as $\omega_i = u^*_{i,i}$ for all $i$ is a feasible solution of \ref{eq:lp-pdb} with value $z_{\sf RLT-PDB}$. Conversely, let $\bm{\omega}^*$ denote an optimal solution of \ref{eq:lp-pdb}. Then the solution $\mathbf{x}, \mathbf{y}, \mathbf{u}$ defined as $x_{i} = y_{i}=u_{ii}=\omega^*_i$ for all $i$ and $u_{i,j} = 0$ for all $i \neq j$ is feasible for \ref{eq:rlt-pdb} (note that $\theta_i \mathbf{P}_i \leq \mathbf{p}$ and $\gamma_i \mathbf{Q}_i \leq \mathbf{q}$ for every $i$ by definition of $\theta_i$ and $\gamma_i$). The value of this solution is $z_{\sf LP-PDB}$.

One might want to add more polynomial constraints to \ref{eq:pdb} in the reformulation phase. For instance, one might want to add all the constraints we get from multiplying the linear constraints of \ref{eq:pdb} by all the first order polynomial terms $x_i$, $y_i$, $(1-x_i)$ and $(1-y_i)$. The resulting relaxation, known as the first level relaxation of the RLT hierarchy, is a widely used LP approximation for bilinear programs and has been observed to be a good empirical approximation (Sherali et al. \cite{sherali2000reduced}, Sherali and Adams \cite{adams1986tight,adams1993mixed}, Sherali and Alameddine \cite{sherali1992new}). We refer the reader to the Appendix \ref{appendix:1RLT} for a derivation of this relaxation and to Section~\ref{section:numerics} for a numerical comparison of this relaxation to \ref{eq:lp-pdb}. One might also want to multiply by higher order polynomial constraints of the form $\Pi(S,S',T,T') = \prod_{i \in S} x_i \prod_{i \in S'} y_i \prod_{i \in T} (1-x_i) \prod_{i \in T'} (1-y_i)$ where $S, S', T, T' \subset [n]$ at the expense of an exponential increase in the number of constraints of the resulting relaxation. The relaxation we get from multiplying by $\Pi(S,S',T,T')$ for all $S,S',T,T'$ such that $|S|+|S'|+|T|+|T'| \leq t$ is known as the $t$-th level relaxation of the RLT hierarchy.

The equivalence between \ref{eq:lp-pdb} and \ref{eq:rlt-pdb} gives a novel perspective on our LP relaxation and shows that the $t$-th level relaxations of the RLT hierarchy are guaranteed to be within $O(\frac{\log m_1}{\log \log m_1} \frac{\log m_2}{\log \log m_2})$ of the optimum for packing disjoint bilinear programs.

\section{From Disjoint Bilinear Optimization to Two-Stage Adjustable Robust Optimization}
\label{section:ar}

In this section, we present a polylogarithmic approximation algorithm for \ref{eq:ar}. In particular, we give a compact linear restriction of \ref{eq:ar} that provides near-optimal first-stage solutions with cost that is within a factor of $O(\frac{\log n}{\log\log n}\frac{\log L}{\log\log L})$ of $z_{\sf{AR}}$. Our proof uses ideas from our approximation of \ref{eq:pdb} applied to the separation problem $\hyperref[eq:qx]{\mathcal{Q}(\mathbf{x})}$. 

In order to simplify the exposition, we make the following assumption in the remainder of the paper,
\begin{assumption}
\label{assumption}
    $\mathbf{A}\mathbf{x} \geq 0$ for every $\mathbf{x} \in \mathcal{X}$.
\end{assumption}
\noindent Assumption \ref{assumption} states that first stage solutions can only help cover the whole or part of the uncertain right hand side uncertain parameters $\mathbf{h}$. We show in Appendix \ref{appendix:generalA} that Assumption \ref{assumption} is without loss of generality. We also give in Appendix \ref{appendix:generalA} an example of a network design application where Assumption \ref{assumption} does not hold.

Recall the two-stage adjustable problem \ref{eq:ar},
\begin{gather}
\tag*{ }
\begin{aligned}
\min_{\mathbf{x} \in \mathcal{X}} \quad \mathbf{c}^T\mathbf{x} + \mathcal{Q}(\mathbf{x}),
\end{aligned}
\end{gather}
where for all $\mathbf{x} \in \mathcal{X}$,
\begin{gather}
\tag*{ }
\begin{aligned}
\mathcal{Q}(\mathbf{x}) = \max_{\mathbf{h} \in \mathcal{U}} \min_{\mathbf{y}\geq \mathbf{0}} \; \{\mathbf{d}^T\mathbf{y} \mid
\mathbf{A}\mathbf{x} + \mathbf{B}\mathbf{y}\geq \mathbf{h}\}.
\end{aligned}
\end{gather}
Let us write $\hyperref[eq:qx]{\mathcal{Q}(\mathbf{x})}$ in its bilinear form. In particular, we take the dual of the inner minimization problem on $\mathbf{y}$ to get,
\begin{gather}
\tag*{ }
\begin{aligned}
\mathcal{Q}(\mathbf{x}) = \max_{\mathbf{h}, \mathbf{z} \geq \mathbf{0}} \; \left\{ \;\;
\mathbf{h}^T\mathbf{z} - (\mathbf{A}\mathbf{x})^T\mathbf{z} \;\; \left| \;\;\begin{aligned}
\mathbf{B}^T\mathbf{z} &\leq \mathbf{d}\\
\mathbf{R}\mathbf{h} &\leq \mathbf{r}\;\;
\end{aligned}
\right.\right\}.
\end{aligned}
\end{gather}

For the special case where $\mathbf{A} = \mathbf{0}$, the optimal first-stage solution is $\mathbf{x} = \mathbf{0}$ and \ref{eq:ar} reduces to an instance of \ref{eq:pdb}. 
Therefore, our algorithm for \ref{eq:pdb} gives an $O(\frac{\log n}{\log \log n}\frac{\log L}{\log \log L})$-approximation algorithm of \ref{eq:ar} in this special case.

In the general case, the separation problem $\hyperref[eq:qx]{\mathcal{Q}(\mathbf{x})}$ is the difference of a bilinear and a linear term. This makes it challenging to approximate $\hyperref[eq:qx]{\mathcal{Q}(\mathbf{x})}$. Instead, we approximate \ref{eq:ar} directly. In particular, for any $\mathbf{x} \in \mathcal{X}$, consider the following linear program:
\begin{gather}
\label{eq:qlp}
\tag*{ }
\mathcal{Q}^{\sf{LP}}(\mathbf{x}) = \max_{\bm{\omega} \geq \mathbf{0}} \left\{\;\; \sum_{i=1}^m (\theta_i \gamma_i - \theta_i\mathbf{a}_i^T\mathbf{x}) \omega_i \;\; \left| \;\; \begin{aligned}
\sum_{i=1}^m \theta_i \mathbf{b}_i \omega_i &\leq \mathbf{d}\\
\sum_{i=1}^m \gamma_i \mathbf{R}_i \omega_i &\leq \mathbf{r}
\end{aligned}
\right.\right\},
\end{gather}
where for all $i \in [m]$, $$\theta_i := \max_{\mathbf{z}}\{z_i \mid \mathbf{B}^T\mathbf{z} \leq \mathbf{d}, \mathbf{z} \geq 0\}, \quad \gamma_i := \max_{\mathbf{h}}\{h_i \mid \mathbf{R}\mathbf{h} \leq \mathbf{r}, \mathbf{h} \geq 0\},$$ $\mathbf{a}_i$ and $\mathbf{b}_i$ are the $i$-th row of $\mathbf{A}$ and $\mathbf{B}$ respectively and $\mathbf{R}_i$ is the $i$-th column of $\mathbf{R}$. For ease of notation, we use $\mathcal{Q}^{\sf{LP}}(\mathbf{x})$ to refer to both the problem and its optimal value. Let $$\eta := \frac{3\log n}{\log \log n} + 2, \quad \quad \beta := \frac{3\log L}{\log \log L} + 2.$$ Similar to \ref{eq:pdb}, we show the following structural property of the separation problem.

\begin{lemma} {\normalfont (\textbf{Structural Property})}.
\label{lm:structural-property-ar}
For every $\mathbf{x} \in \mathcal{X}$, there exists a near-integral solution $(\mathbf{h}, \mathbf{z}) \in \{0, \frac{\gamma_i}{\beta}\}^m \times \{0, \frac{\theta_i}{\eta}\}^m$ of $\hyperref[eq:qx]{\mathcal{Q}(\mathbf{x})}$ such that,
\begin{gather}
\label{eq:properties-ar}
\begin{aligned}
    \sum_{i=1}^m \mathbf{b}_i z_i &\leq \mathbf{d},\\
    \sum_{i=1}^m \mathbf{R}_i h_i &\leq \mathbf{r},\\
    \sum_{i=1}^m h_i z_i - (\mathbf{a}_i^T\mathbf{x})z_i &\geq \frac{1}{2\eta\beta} \cdot \mathcal{Q}^{\sf{LP}}(\beta\mathbf{x}).
\end{aligned}
\end{gather}
\end{lemma}
\noindent We construct such solution following a similar procedure as in Algorithm~\ref{alg:rounding-pdb}. In particular, let $\bm{\omega}^*$ be an optimal solution of $\hyperref[eq:qlp]{\mathcal{Q}^{\sf{LP}}(\beta\mathbf{x})}$, consider $\Tilde{\omega}_1, \dots, \Tilde{\omega}_m$ i.i.d. Bernoulli random variables such that $\mathbb{P}(\Tilde{\omega}_i = 1) = \omega^*_i$ for all $i \in [m]$ and let $(\mathbf{h}, \mathbf{z})$ such that $h_i = \frac{\gamma_i\Tilde{\omega}_i}{\beta}$ and $z_i = \frac{\theta_i\Tilde{\omega}_i}{\eta}$ for all $i \in [m]$. Such $(\mathbf{h}, \mathbf{z})$ satisfies the properties~\eqref{eq:properties-ar} with a constant probability. The proof of this fact is similar to the proof of Lemma~\ref{lm:structural-property-pdb} and is deferred to Appendix \ref{appendix:proof_structural_lemma}.

Note that while the linearization $\hyperref[eq:qlp]{\mathcal{Q}^{\sf{LP}}(\mathbf{x})}$ is still an upper-bound of $\hyperref[eq:qx]{\mathcal{Q}(\mathbf{x})}$ (as in the case of \ref{eq:pdb}), we cannot guarantee anymore that $\hyperref[eq:qlp]{\mathcal{Q}^{\sf{LP}}(\mathbf{x})}$ gives a good lower-bound of $\hyperref[eq:qx]{\mathcal{Q}(\mathbf{x})}$ due to the linear term. To see why, consider an optimal solution $\bm{\omega}^*$ of $\hyperref[eq:qlp]{\mathcal{Q}^{\sf{LP}}(\mathbf{x})}$ and let us try, as in the case of \ref{eq:pdb}, to construct a solution of $\hyperref[eq:qx]{\mathcal{Q}(\mathbf{x})}$ whose objective value is at least a fraction of the objective value of $\bm{\omega}^*$ in $\hyperref[eq:qlp]{\mathcal{Q}^{\sf{LP}}(\mathbf{x})}$. In particular, round $\bm{\omega}^*$ to a $0-1$ vector that we refer to as $\Tilde{\bm{\omega}}$, then take $(\mathbf{h}, \mathbf{z})$ such that $h_i = \frac{\gamma_i\Tilde{\omega}_i}{\beta}$ and $z_i = \frac{\theta_i\Tilde{\omega}_i}{\eta}$ for all $i \in [m]$. The objective value of $(\mathbf{h}, \mathbf{z})$ in $\hyperref[eq:qx]{\mathcal{Q}(\mathbf{x})}$ is now $\frac{1}{2\eta\beta}(\sum_{i=1}^m \theta_i\gamma_i\tilde{\omega}_i - \beta\theta_i \mathbf{a}_i^T\mathbf{x}\tilde{\omega}_i)$. This cannot be directly related to the objective value of $\tilde{\bm{\omega}}$ (and therefore to the objective value of $\bm{\omega}^*$) in $\hyperref[eq:qlp]{\mathcal{Q}^{\sf{LP}}(\mathbf{x})}$ due to an extra $\beta$ factor in the linear term. It is therefore unclear how to lower-bound $\hyperref[eq:qx]{\mathcal{Q}(\mathbf{x})}$ using $\hyperref[eq:qlp]{\mathcal{Q}^{\sf{LP}}(\mathbf{x})}$. To get around this issue, we use $\hyperref[eq:qlp]{\mathcal{Q}^{\sf{LP}}(\beta\mathbf{x})}$ to lower-bound $\hyperref[eq:qx]{\mathcal{Q}(\mathbf{x})}$ instead. We show that this is enough for our purposes given that when $\mathbf{x}$ is feasible for \ref{eq:ar} then $\beta\mathbf{x}$ is also feasible and vice versa. The above discussion is formalized in the following lemma.

\begin{lemma}
\label{lm:qx}
For  $\mathbf{x} \in \mathcal{X}$, we have,
\begin{align*}
    \frac{1}{2\eta \beta}\mathcal{Q}^{\sf{LP}}(\beta\mathbf{x}) \leq \mathcal{Q}(\mathbf{x}) \leq \mathcal{Q}^{\sf{LP}}(\mathbf{x}).
\end{align*}
\end{lemma}

\vspace{2mm}\noindent {\em Proof.} First, let $(\mathbf{h}^*, \mathbf{z}^*)$ be an optimal solution of $\hyperref[eq:qx]{\mathcal{Q}(\mathbf{x})}$. 
Define $\bm{\omega}^*$ such that $\omega^*_i = \frac{z^*_i}{\theta_i}.\frac{h^*_i}{\gamma_i}$ for all $i \in [m]$. Then $\bm{\omega}^*$ is feasible for $\hyperref[eq:qlp]{\mathcal{Q}^{\sf{LP}}(\mathbf{x})}$ with objective value, 
\begin{align*}
    \sum_{i=1}^{m} (\theta_i \gamma_i - \theta_i\mathbf{a}_i^T\mathbf{x} ) \omega^*_i 
    &= \sum_{i=1}^{m} h^*_i z^*_i - (\mathbf{a}_i^T\mathbf{x}) \frac{h^*_i}{\gamma_i} z^*_i\\ &\geq \sum_{i=1}^{m} h^*_i z^*_i - (\mathbf{a}_i^T\mathbf{x}) z^*_i \\
    &= \mathcal{Q}(\mathbf{x})
\end{align*}
where the inequality follows from the fact that $\frac{h_i}{\gamma_i} \leq 1$ for all $i \in [m]$. Hence $\mathcal{Q}(\mathbf{x}) \leq \mathcal{Q}^{\sf{LP}}(\mathbf{x})$.

Now, consider $(\mathbf{h}, \mathbf{z}) \in \{0, \frac{\gamma_i}{\beta}\}^m \times \{0, \frac{\theta_i}{\eta}\}^m$ satisfying properties~\eqref{eq:properties-ar}. The first two properties imply that $(\mathbf{h}, \mathbf{z})$ is a feasible solution for $\hyperref[eq:qx]{\mathcal{Q}(\mathbf{x})}$. The third property implies that the objective value of this solution is such that,
$$\sum_{i=1}^{m} h_i z_i - \mathbf{a}_i^T\mathbf{x} z_i
\geq \frac{1}{2\eta\beta} \mathcal{Q}^{\sf{LP}}(\beta\mathbf{x}).$$ Hence, $\mathcal{Q}(\mathbf{x}) \geq \frac{1}{2\eta\beta}\mathcal{Q}^{\sf{LP}}(\beta\mathbf{x})$. \qed

\vspace{2mm}
\noindent {\bf Our Linear Restriction}. We are now ready to derive our linear restriction of \ref{eq:ar}. In particular, consider the following problem where $\hyperref[eq:qx]{\mathcal{Q}(\mathbf{x})}$ is replaced by $\hyperref[eq:qlp]{\mathcal{Q}^{\sf{LP}}(\mathbf{x})}$ in the expression of \ref{eq:ar},
\begin{gather}\label{eq:lp-ar-0}
z_{\sf{LP-AR}} = \min_{\mathbf x \in \mathcal{X}} \; \left\{ \mathbf{c}^T\mathbf{x} +{\mathcal{Q}^{\sf{LP}}(\mathbf{x})}\right\}.
\end{gather}
Note that for any given first stage solution $\mathbf x$, $\hyperref[eq:qlp]{\mathcal{Q}^{\sf{LP}}(\mathbf{x})}$ is a maximization LP. Taking its dual 
and substituting in~\eqref{eq:lp-ar-0}, we get the following LP:

\begin{gather}
\tag{LP-AR}
\label{eq:lp-ar}
\begin{aligned}
z_{\sf{LP-AR}} = \min_{\mathbf{x}, \mathbf{y}, \bm{\alpha}} \quad & \mathbf{c}^T\mathbf{x} + \mathbf{d}^T\mathbf{y} + \mathbf{r}^T\bm{\alpha}\\
\textrm{s.t.} \quad & \theta_i \mathbf{a}_i^T\mathbf{x} + \theta_i \mathbf{b}_i^T\mathbf{y} + \gamma_i \mathbf{R}_i^T\bm{\alpha} \geq \theta_i\gamma_i \quad \forall i,\\
\quad & \mathbf{x} \in \mathcal{X}, \mathbf{y} \geq \mathbf{0}, \bm{\alpha} \geq \mathbf{0}.
\end{aligned}
\end{gather}
We claim that \ref{eq:lp-ar} is a restriction of \ref{eq:ar} and gives an $O(\frac{\log n}{\log \log n}\frac{\log L}{\log \log L})$-approximation for \ref{eq:ar}.

\vspace{2mm}\noindent {\em Proof of Theorem~\ref{th:ar}.} We prove the following:
\begin{align*}
    z_{\sf{AR}} \leq z_{\sf{LP-AR}} \leq 3\eta\beta z_{\sf{AR}}.
\end{align*}
Let $\mathbf{x}_{\sf LP}^*$ denote an optimal solution of (\ref{eq:lp-ar-0}). We have,
\begin{align*}
z_{\sf{LP-AR}}
&= \mathbf{c}^T\mathbf{x}_{\sf LP}^* + \mathcal{Q}^{\sf{LP}}(\mathbf{x}^*_{\sf{LP}})\\ 
&\geq \mathbf{c}^T\mathbf{x}_{\sf LP}^* + \mathcal{Q}(\mathbf{x}_{\sf LP}^*)\\
&\geq z_{\sf{AR}},
\end{align*}
where the first inequality follows from Lemma~\ref{lm:qx} and the last inequality follows from the feasibility of $\mathbf{x}^*_{\sf{LP}}$ in \ref{eq:ar}.

To prove the upper-bound on $z_{\sf{LP-AR}}$, let $\mathbf{x}^*$ denote an optimal first-stage solution of \ref{eq:ar}. We have
\begin{align*}
    z_{\sf{AR}}
    &\geq
    \mathcal{Q}(\mathbf{x}^*)\\
    &\geq
    \frac{1}{2\eta\beta} \mathcal{Q}^{\sf{LP}}(\beta \mathbf{x}^*)\\
    &=
    \frac{1}{2\eta\beta}(\beta \mathbf{c}^T\mathbf{x}^* + \mathcal{Q}^{\sf{LP}}(\beta\mathbf{x}^*)) - \frac{1}{2\eta} \mathbf{c}^T\mathbf{x}^*\\
    &\geq
    \frac{1}{2\eta\beta}(\beta \mathbf{c}^T\mathbf{x}^* + \mathcal{Q}^{\sf{LP}}(\beta\mathbf{x}^*)) - \frac{1}{2}z_{\sf{AR}}\\
    &\geq
    \frac{1}{2\eta\beta}z_{\sf{LP-AR}} - \frac{1}{2}z_{\sf{AR}},
\end{align*}
where the second inequality follows from Lemma~\ref{lm:qx}, the third inequality follows from the fact that $\mathbf{c}^T\mathbf{x}^* \leq z_{\sf AR}$ and the fact that $\eta \geq 1$. For the last inequality, note that $\beta \mathbf{x}^*$ is a feasible solution for (\ref{eq:lp-ar-0}).The above implies that $z_{\sf{LP-AR}} \leq 3\eta\beta z_{\sf{AR}}$. \qed

\section{Affine Policies and Relation to our LP Restriction} 
\label{section:affine}

Affine policies are widely used to approximate the two-stage problem \ref{eq:ar} where the second-stage variables $\mathbf{y}$ are restricted to be affine functions of the uncertain right-hand side $\mathbf{h}$. In other words, we consider $\mathbf{y}(\mathbf{h}) = \mathbf{P}\mathbf{h}+\mathbf{q},$ and optimize over $\mathbf{P}$ and $\mathbf{q}$. Affine policies were first introduced in Ben-Tal et al.~\cite{bental2004} and have been widely considered in the literature. Ben-Tal et al.~\cite{bental2004} show that affine policies give a tractable approximation for a large class of dynamic optimization problems. In particular, for a polyhedral uncertainty set $\mathcal{U}$, one can find the optimal affine policy by solving a linear program with polynomially many variables and constraints.

Many approximation bounds are known for the worst-case performance of affine policies in many settings. For a simplex uncertainty set, Bertsimas and Goyal \cite{Bertsimas2012OnTP} show that affine policies are optimal. In the case of a single budget of uncertainty set, El Housni and Goyal \cite{housni2021affine} show that affine policies achieve an approximation bound of $O(\frac{\log n}{\log \log n})$. They also show an $O(\frac{\log^2n}{\log\log n})$ approximation bound in the case of an intersection of disjoint budget of uncertainty sets (partition matroid) and an $O(\frac{L\log n}{\log \log n})$ approximation bound in the case of a general intersection of $L$ budget of uncertainty sets. The later results rely on a clever decomposition of the coordinates of $\mathbf{h}$ into cheap and expensive coordinates. In this section, we generalize the bound for a single budget of uncertainty set and significantly improve the bounds for partition matroids and general intersections of budget of uncertainty sets. In particular, we show that affine policies give an $O(\frac{\log n}{\log \log n}\frac{\log L}{\log \log L})$ worst-case approximation to \ref{eq:ar} under the general packing uncertainty set with $L$ constraints. In contrast with the techniques used in El Housni and Goyal \cite{housni2021affine}, our proof uses our LP restriction for \ref{eq:ar}. In particular, using an optimal solution of \ref{eq:lp-ar}, we construct an explicit feasible affine policy with objective value at most $z_{\sf LP-AR}$.

\vspace{2mm} \noindent {\em Proof of Theorem~\ref{th:AFF}.} Our construction is as follows: Let $\mathbf{x}^*, \mathbf{y}^*, \bm{\alpha}^*$ be an optimal solution of \ref{eq:lp-ar}. Consider the affine policy with first-stage solution $\mathbf{x}^*$ and second-stage policy
\begin{align}
\label{eq:affine_policy}
    \mathbf{y}_{\sf{AFF}}(\mathbf{h}) := \mathbf{y}^* + \sum_{i} \frac{\mathbf{R}_i^T \bm{\alpha}^*}{\theta_i} \mathbf{v_i}h_i,
\end{align}
for every $\mathbf{h} \in \mathcal{U}$, where for every $i \in [m]$, $$\mathbf{v_i} \in \argmin_{\mathbf{y} \geq 0}\{\mathbf{d}^T\mathbf{y} \mid \mathbf{B} \mathbf{y} \geq \mathbf{e}_i\}.$$ 
Let us first show that the above construction gives a feasible affine policy. We know $\mathbf{x}^* \in \mathcal{X}$ by construction. We need to verify the feasibility of the second-stage solution $\mathbf{y}_{\sf{AFF}}(\mathbf{h})$ for every $\mathbf{h} \in \mathcal{U}$. Consider any $\mathbf{h} \in \mathcal{U}$, the $i$-th constraint of second-stage problem is given by,

\begin{align*}
    \mathbf{a}_i^T\mathbf{x}^* + \mathbf{b}_i^T\mathbf{y}_{\sf{AFF}}(\mathbf{h}) 
    &= \mathbf{a}_i^T\mathbf{x}^*  + \mathbf{b}_i^T\mathbf{y}^* + \sum_{k=1}^m \frac{\mathbf{R}_k^T \bm{\alpha}^*}{\theta_k} \mathbf{b}_i^T\mathbf{v_k}h_k\\
    &\geq \mathbf{a}_i^T\mathbf{x}^*  + \mathbf{b}_i^T\mathbf{y}^* + \frac{\mathbf{R}_i^T \bm{\alpha}^*}{\theta_i} \mathbf{b}_i^T\mathbf{v_i}h_i\\
    &\geq \mathbf{a}_i^T\mathbf{x}^*  + \mathbf{b}_i^T\mathbf{y}^* + \frac{\mathbf{R}_i^T \bm{\alpha}^*}{\theta_i}h_i\\
    &\geq (\mathbf{a}_i^T\mathbf{x}^* + \mathbf{b}_i^T\mathbf{y}^*)\frac{h_i}{\gamma_i} + \frac{\mathbf{R}_i^T \bm{\alpha}^*}{\theta_i}h_i\\
    &= \frac{1}{\gamma_i\theta_i}(\theta_i\mathbf{a}_i^T\mathbf{x}^* + \theta_i\mathbf{b}_i^T\mathbf{y}^* + \gamma_i\mathbf{R}_i^T \bm{\alpha}^*)h_i,\\
    &\geq h_i,
\end{align*}

\noindent where the first inequality follows from the fact that $\frac{\mathbf{R}_k^T \bm{\alpha}^*}{\theta_k} \mathbf{b}_i^T\mathbf{v_k}h_k \geq 0$ for every $k \in [n]$. The second inequality follows from the definition of $\mathbf{v}_i$. The third inequality follows from the non-negativity of $(\mathbf{a}_i^T\mathbf{x}^* + \mathbf{b}_i^T\mathbf{y}^*)$ and because $h_i \leq \gamma_i$. And the last inequality follows from the feasibility of $\mathbf{x}^*, \mathbf{y}^*$ and $\bm{\alpha}^*$ in \ref{eq:lp-ar}.

Next, we show that the objective value of our affine policy is upper-bounded by the $z_{\sf{LP-AR}}$. Consider any $\mathbf{h} \in \mathcal{U}$. The cost of our affine policy under the scenario $\mathbf{h}$ is given by,
\begin{align*}
    \mathbf{c}^T\mathbf{x}^* + \mathbf{d}^T\mathbf{y}_{\sf{AFF}}(\mathbf{h}) 
    &= \mathbf{c}^T\mathbf{x}^*  + \mathbf{d}^T\mathbf{y}^* + \sum_{k=1}^m \frac{\mathbf{R}_k^T \bm{\alpha}^*}{\theta_k} \mathbf{d}^T\mathbf{v_k}h_k\\
    &= \mathbf{c}^T\mathbf{x}^* + \mathbf{d}^T\mathbf{y}^* + \sum_{k=1}^m \mathbf{R}_k^T \bm{\alpha}^* h_k\\
    &\leq \mathbf{c}^T\mathbf{x}^* + \mathbf{d}^T\mathbf{y}^* + \mathbf{r}^T \bm{\alpha}^* \\
    &= z_{\sf{LP-AR}},
\end{align*}
where the second equality follows from the fact that $\theta_k = \mathbf{d}^T\mathbf{v}_k$ by definition of $\theta_k$ and $\mathbf{v}_k$ and by strong duality. The inequality follows from the fact that $\sum_{k=1}^m \mathbf{R}_k h_k = \mathbf{R}\mathbf{h} \leq \mathbf{r}$. Therefore, the worst-case cost of our affine policy is at most $z_{\sf{LP-AR}}$.
\qed

\section{Numerical Experiments}
\label{section:numerics}

In this section, we study the numerical performance of our approximation algorithms for \ref{eq:pdb} and \ref{eq:ar}. We compare our algorithms to several state-of-the art algorithms
for these problems and observe that our algorithms are significantly faster and provide good approximate solutions. 

\subsection{Performance of our Approximation for \ref{eq:pdb}}

Our randomized rounding based algorithm for~\ref{eq:pdb}  gives an $O(\frac{\log \log m_1}{\log m_1} \frac{\log \log m_2}{\log m_2})$-approximation with high probability. We would like to note that this is a worst-case guarantee and our goal in this section is to evaluate the empirical performance of our algorithm. 

Before presenting the numerical evaluation of our algorithm, we would like to note that our algorithm can be numerically improved in the following natural way.

\vspace{2mm}
\noindent {\bf An Improved Version of Algorithm~\ref{alg:rounding-pdb}.} Let $(\hat{\mathbf{x}}, \hat{\mathbf{y}})$ denote the solution output by Algorithm \ref{alg:rounding-pdb}. We consider a natural heuristic to improve the objective value of our solution which consists of performing an alternating maximization starting from $(\hat{\mathbf{x}}, \hat{\mathbf{y}})$. In particular, we consider the algorithm which computes a sequence $(\mathbf{x}^1, \mathbf{y}^1), (\mathbf{x}^2, \mathbf{y}^2), \dots$ of improving solutions of \ref{eq:pdb} as follows: Let $(\mathbf{x}^1, \mathbf{y}^1) = (\Tilde{\mathbf{x}}, \Tilde{\mathbf{y}})$. To compute the improved solution $(\mathbf{x}^{i+1}, \mathbf{y}^{i+1})$ from $(\mathbf{x}^{i}, \mathbf{y}^{i})$, we fix the variables $\mathbf{x}$ in \ref{eq:pdb} to $\mathbf{x}^i$ and maximize with respect to $\mathbf{y}$ to get $\mathbf{y}^{i+1}$. Then we fix the variables $\mathbf{y}$ to $\mathbf{y}^{i+1}$ and maximize with respect to $\mathbf{x}$ to get $\mathbf{x}^{i+1}$. Note that each step is a linear maximization problem. The algorithm stops when the change in the objective value between two consecutive solutions $(\mathbf{x}^{i}, \mathbf{y}^{i})$ and $(\mathbf{x}^{i+1}, \mathbf{y}^{i+1})$ is at most $\epsilon$ (where $\epsilon>0$ is some chosen error tolerance). We refer to this improved version of Algorithm~\ref{alg:rounding-pdb} as Algorithm $2$.





We evaluate the empirical performance of Algorithm~\ref{alg:rounding-pdb} and its improved version Algorithm $2$ by comparing these to the following benchmarks.

\vspace{2mm} \noindent \textbf{First Level Relaxation of the RLT Hierarchy.} The first level relaxation of the RLT hierarchy is a widely used LP approximation for bilinear programs and has been observed to be a good empirical approximation (Sherali et al. \cite{sherali2000reduced}, Sherali and Adams \cite{adams1986tight,adams1993mixed}, Sherali and Alameddine \cite{sherali1992new}). 
We denote the relaxation by \ref{eq:1rlt} 
(we refer the reader to Appendix \ref{appendix:1RLT} for the derivation of this relaxation). 
An optimal solution $(\mathbf{x}, \mathbf{y}, \mathbf{u})$ of \ref{eq:1rlt} gives an approximate solution $(\mathbf{x}, \mathbf{y})$ for \ref{eq:pdb}. 
We compare this solution with the solution given by our algorithms in terms of objective value and running time. 


\vspace{2mm} \noindent \textbf{Gurobi Solver.} We also compare the solution given by our algorithms with the optimal solution of \ref{eq:pdb} computed using the bilinear solver of Gurobi v9.1.2, in terms of both objective value and running time.
In particular, we compare the running time of our algorithms with the running time needed by the Gurobi solver to reach a solution that is at least as good as the worst of our two solutions, namely the solution given by Algorithm~\ref{alg:rounding-pdb}. 

\paragraph{Experimental Setup.}
We consider instances of \ref{eq:pdb} where $m_1=m_2$, $\mathbf{p}=\mathbf{q}=\mathbf{e}$, $\mathbf{P}=\mathbf{I}_{m}+\mathbf{G}_{P}$ and $\mathbf{Q}=\mathbf{I}_{m}+\mathbf{G}_{Q}$ where $\mathbf{I}_m$ is the identity matrix and $\mathbf{G}_{P}$
and $\mathbf{G}_{Q}$ are random normalized Gaussian matrices. The results of our experiments were performed on a dual-core laptop with 8GB of RAM and 1.8GHz processor. 

Table \ref{table:lb-pdb} gives the results of our comparison between the solutions given by Algorithm \ref{alg:rounding-pdb} and its improved version Algorithm $2$ and the solutions given by the first level relaxation of the RLT hierarchy and the bilinear solver of Gurobi v9.1.2. In Table~\ref{table:lb-pdb}, $T_{\sf ALG_1}$ denotes the running time in seconds of Algorithm~\ref{alg:rounding-pdb}, $T_{\sf ALG_2}$ the running time in seconds of Algorithm $2$, $T_{\sf FL-RLT}$ the running time in seconds of the first level relaxation of the RLT hierarchy, $T_{\sf GRB}$ the running time in seconds of Gurobi solver to solve the problem to optimality and $T_{\sf GRB-LB}$ the running time in seconds needed by the Gurobi solver to reach a solution that is as good as the solution given by Algorithm~\ref{alg:rounding-pdb}. For the objective values, $z_{\sf SOL-RLT}$ denotes the objective value of the solution given by the first level relaxation of the RLT hierarchy, $z_{\sf ALG_1}$ and $z_{\sf ALG_2}$ the objective value of the solution given by Algorithm~\ref{alg:rounding-pdb} and Algorithm $2$ respectively and $z_{\sf PDB}$ denotes the optimal objective.

Note that Algorithm~\ref{alg:rounding-pdb} is based on the linear relaxation \ref{eq:lp-pdb} of \ref{eq:pdb}. The objective value of this relaxation gives an upper-bound of~\ref{eq:pdb}. The first level relaxation of the RLT hierarchy \ref{eq:1rlt} is also a relaxation of \ref{eq:pdb} and therefore, also gives an upper-bound of~\ref{eq:pdb}. 
The bilinear solver of Gurobi gives a sequence of improving upper-bounds of~\ref{eq:pdb}. We compare the upper-bound given by our LP relaxation denoted by $z_{\sf LP-PDB}$ and the upper-bound given by the first level relaxation of the RLT hierarchy denoted by $z_{\sf FL-RLT}$. We also compare the running time  to compute $z_{\sf LP-PDB}$, the running time  to compute $z_{\sf FL-RLT}$ and the running time  needed by the bilinear solver of Gurobi to reach an upper-bound that is at most $z_{\sf LP-PDB}$ denoted by $T_{\sf LP-PDB}$, $T_{\sf FL-RLT}$ and $T_{\sf GRB-UB}$ respectively. The results of the comparisons are reported in Table \ref{table:ub-pdb}.



\paragraph{Discussion.}
Table \ref{table:lb-pdb} describes the results of our comparison between the solutions given by Algorithm \ref{alg:rounding-pdb} and its improved version Algorithm $2$ and the solutions given by the first level relaxation of the RLT hierarchy and the bilinear solver of Gurobi v9.1.2 on random instances and for different values of $m_1=m_2$.

Compared to the first level relaxation of the RLT hierarchy, our algorithms are significantly faster especially for large dimensions. For example, for $m_1=m_2 \geq 100$, Algorithm $2$ is more than $400$ times faster while Algorithm~\ref{alg:rounding-pdb} is more than $900$ times faster. In terms of objective value, the solutions given by our algorithms have higher objective value than the solutions given by the first level relaxation of the RLT hierarchy. Moreover, the gap in the objective value grows with the dimension. For instance, for $m_1=m_2 \geq 40$, the objective value of the solution given by \ref{eq:1rlt} is less $40\%$ of the objective value of the solution given by Algorithm~\ref{alg:rounding-pdb} and less than $34\%$ of the objective value of the solution given by Algorithm $2$. 

We would like to note that compared to the bilinear solver of Gurobi, the running times of our algorithms  are significantly faster. In fact, computing an optimal solution using Gurobi is impractical for higher dimensions (Gurobi solver fails to finish in three hours for $m_1=m_2 > 40$). We also remark that, as the dimension increases, the bilinear solver of Gurobi takes significantly more time to compute a solution that is comparable to the solutions from our algorithms. For example, for $m_1=m_2\geq 500$, the bilinear solver of Gurobi is at least $1000$ times slower. Furthermore, we observe that the objective value of our solutions is a good approximation of the optimal value. For example, for all $m_1=m_2 \leq 40$, the solution given by Algorithm~\ref{alg:rounding-pdb} recovers at least $76 \%$ of the optimum while the solution given by Algorithm $2$ recovers at least $88 \%$ of the optimum.
Therefore, our algorithm and its improved version give a faster approach to compute a near-optimal approximate solution of \ref{eq:pdb}. In particular, this could be used as a more efficient approach to compute a good feasible solution in branch-and-bound based exact algorithms for the problem.






Table \ref{table:ub-pdb} gives the results of our comparison between the upper-bounds given by \ref{eq:lp-pdb}, the first level relaxation of the RLT hierarchy and the bilinear solver of Gurobi v9.1.2 on random instances and for different values of $m_1=m_2$. Table \ref{table:ub-pdb} shows that, as the dimension increases, our upper-bound $z_{\sf LP-PDB}$ becomes significantly faster to compute than the first level relaxation of the RLT hierarchy. For example, for $m_1=m_2 \geq 100$ our upper-bound is more than $3700$ times faster to compute. Moreover, the ratio between the two upper-bounds is close to $1$ and increases with the dimension (our upper-bound is at least $77\%$ of $z_{\sf FL-RLT}$ for all the instances we consider; it is at least $86\%$ for $m_1=m_2=100$). Compared to Gurobi, and as the dimension increases, the bilinear solver of Gurobi takes significantly more time to compute an upper-bound that is at least $z_{\sf LP-PDB}$. For example, for $m_1=m_2 \geq 500$, the bilinear solver of Gurobi is more than $7$ times slower. Given the above discussion, we conclude that our relaxation gives a faster to compute near-optimal upper-bound of \ref{eq:pdb} which can be used in branch-and-bound type of algorithms for the problem.


\begin{table}
    \centering

    
    \caption{Comparison of Algorithm~\ref{alg:rounding-pdb} and its improved version Algorithm $2$ with the first level relaxation of RLT hierarchy and the bilinear solver of Gurobi v9.1.2 for different values of $m_1=m_2$. The running times are in seconds and the entries denoted by $*$ exceeded a time limit of three hours.}
    \centering
    \begin{tabular}{|c|c|c|c|c|c|}
     \hline
     &&&&&\\
     $m_1=m_2$ & $T_{\sf ALG_1}$ & $T_{\sf ALG_2}$ & $T_{\sf FL-RLT}$ & $T_{\sf GRB}$ & $T_{\sf GRB-LB}$ \\[2ex]
     \hline
     10 & 0.464 & 0.686 & 0.223 & 0.356 & 0.219\\
     \hline
     20 & 0.323 & 0.474 & 0.715 & 1.993 & 0.271\\
     \hline
     30 & 0.185 & 0.342 & 1.729 & 60.667 & 0.228 \\
     \hline
     40 & 0.210 & 0.369 & 5.080 & 2027.430 & 1.328 \\
     \hline
     50 & 0.210 & 0.373 & 10.050 & * & 2.083\\
     \hline
     100 & 0.321 & 0.643 & 318.119 & * & 27.282\\
     \hline
     500 & 5.681 & 9.860 & * & * & 7850.564\\
     \hline
     1000 & 17.590 & 31.2890 & * & * & *\\
     \hline
    \end{tabular}
    
    \vspace{3mm}
    
    \begin{tabular}{|c|c|c|c|c|c|}
     \hline
     &&&&&\\
     $m_1=m_2$ & $\frac{z_{\sf SOL-RLT}}{z_{\sf ALG_1}}$ & $\frac{z_{\sf SOL-RLT}}{z_{\sf ALG_2}}$ & $\frac{z_{\sf ALG_1}}{z_{\sf PDB}}$ & $\frac{z_{\sf ALG_2}}{z_{\sf PDB}}$ & $\frac{z_{\sf ALG_1}}{z_{\sf ALG_2}}$ \\[2ex]
     \hline
     10 & 0.816 & 0.763 & 0.913 & 0.976 & 0.935\\
     \hline
     20 & 0.451 & 0.418 & 0.841 & 0.907 & 0.927\\
     \hline
     30 & 0.400 & 0.354 & 0.805 & 0.910 & 0.885\\
     \hline
     40 & 0.381 & 0.331 & 0.768 & 0.883 & 0.870\\
     \hline
     50 & 0.374 & 0.332 & * & * & 0.887\\
     \hline
     100 & 0.329 & 0.280 & * & * & 0.850\\
     \hline
     500 & * & * & * & * & 0.808\\
     \hline
     1000 & * & * & * & * &  0.835\\
     \hline
    \end{tabular}
    \label{table:lb-pdb}
\end{table}

\begin{table}
    \centering
    \caption{Comparison between the upper-bounds given by \ref{eq:lp-pdb}, by the first level relaxation of the RLT hierarchy and by the bilinear solver of Gurobi v9.1.2 for different values of $m_1=m_2$. The running times are in senconds and the entries denoted by $*$ exceeded a time limit of three hours.}
    \centering
        \begin{tabular}{|c|c|c|c|c|c|}
     \hline
     &&&&& \\
     $m_1=m_2$ & $T_{\sf LP-PDB}$ & $T_{\sf FL-RLT}$ & $T_{\sf GRB}$ & $T_{\sf GRB-UB}$ & $\frac{z_{\sf LP-PDB}}{z_{\sf FL-RLT}}
     $ \\[2ex]
     \hline
     10 & 0.084 & 0.223 & 0.356 & 0.086 & 0.779\\
     \hline
     20 & 0.034 & 0.715 & 1.993 & 0.053 & 0.789 \\
     \hline
     30 & 0.022 & 1.729 & 60.667 & 0.074 & 0.803\\
     \hline
     40 & 0.034 & 5.080 & 2027.340 & 0.113 & 0.819 \\
     \hline
     50 & 0.119 & 34.791 & * & 0.293 & 0.834 \\
     \hline
     100 & 0.105 & 394.443 & * & 0.939 & 0.869\\
     \hline
     500 & 2.254 & * & * & 15.570 & *\\
     \hline
     1000 & 9.339 & * & * & 325.815 & * \\
     \hline
    \end{tabular}
    \label{table:ub-pdb}
\end{table}


\subsection{Performance of our Approximation for \ref{eq:ar}}

Recall that our LP restriction of \ref{eq:ar} gives a  polylogarithmic approximation and is tightly related to affine policies for \ref{eq:ar}. In this section, we compare the numerical performance of our LP restriction to several benchmarks including affine policies, a generalization of the algorithm of El Housni and Goyal \cite{housni2021affine}, and the lower-bound of Hadjiyiannis et al. \cite{hadjiyiannis2011scenario}. Let us first briefly discuss the benchmarks.

\vspace{2mm} \noindent \textbf{Affine Policies.} Affine policies are a commonly used approximation for two-stage adjustable robust problems and have been shown to exhibit good theoretical as well as empirical performance (see for example Bertsimas et al. \cite{bertsimas2010optimality}, Bertsimas and Ruiter \cite{Bertsimas2016duality}, Ben-Tal et al. \cite{bental2004} and El Housni and Goyal \cite{GoyalElHoussni2017PAP}). Affine policies consider a restriction of the second-stage variables in \ref{eq:ar} to be affine functions of the uncertain parameters, i.e., functions of the form $\mathbf{y}(\mathbf{h}) = \mathbf{P} \mathbf{h}+\mathbf{q}$ for $\mathbf{P} \in \mathbb{R}^{m \times n}$ and $\mathbf{q} \in \mathbb{R}^n$. The optimal affine policy can be computed using the following restriction of \ref{eq:ar},
\begin{gather*}
\begin{aligned}
    \min_{\mathbf{x}, \mathbf{Y}} \quad & \mathbf{c}^T \mathbf{x} + \max_{\bm{\zeta} \in \hat{\mathcal{U}}}  \mathbf{d}^T\mathbf{Y} \bm{\zeta}\\
    s.t \quad & \mathbf{A}\mathbf{x} + \mathbf{B}\mathbf{Y}\bm{\zeta} \geq \mathbf{C}\bm{\zeta}, \quad \forall \bm{\zeta} \in \hat{\mathcal{U}},
    \\
    & \mathbf{Y} \bm{\zeta} \geq \mathbf{0}, \quad \forall \bm{\zeta} \in \hat{\mathcal{U}},\\
    & \mathbf{x} \in \mathcal{X}
\end{aligned}
\end{gather*}
where $\mathbf{C}$ denotes the matrix $[\mathbf{e}_2, \dots, \mathbf{e}_{m+1}]^T$ with $\mathbf{e}_i$ being the $i$-th vector of the canonical basis of $\mathbb{R}^{m+1}$ and 
$$\hat{\mathcal{U}} = \{\bm{\zeta} \in \mathbb{R}^{m+1} \mid \zeta_1 = 1, \mathbf{C}\bm{\zeta} \in \mathcal{U}\}.$$ 
Let $\mathcal{K}$ be the cone generated by $\hat{\mathcal{U}}$ given as:
$$
\mathcal{K} = \{(t, \mathbf{h}) \in \mathbb{R} \times \mathbb{R}^{m} \mid \mathbf{R} \mathbf{h} \leq \mathbf{r} t,\quad t \geq 0, \mathbf{h} \geq \mathbf{0}\},
$$
and let $\mathcal{K}^*$ be its dual cone given by,
$$
\mathcal{K}^* = \{(u, \bm{\sigma}) \in \mathbb{R} \times \mathbb{R}^{m} \mid \exists \bm{\lambda} \in \mathbb{R}^L \text{ s.t.} \quad u \geq \mathbf{r}^T\bm{\lambda}, \quad \mathbf{R}^T \bm{\lambda} \geq -\bm{\sigma},\quad u \geq 0, \bm{\sigma} \geq \mathbf{0}\}.
$$
For any $\bm{\sigma} \in \mathbb{R}^{m+1}$, 
$$
\bm{\sigma}^T \bm{\zeta} \geq 0 \quad \forall \bm{\zeta} \in \hat{\mathcal{U}} \quad \Leftrightarrow \quad \bm{\sigma} \in \mathcal{K}^*.
$$
The above problem can be expressed as the following equivalent linear program:
\begin{gather}
\label{eq:affine-ar}
\tag{AFF}
\begin{aligned}
    \min_{\mathbf{x}, \alpha} \quad & \mathbf{c}^T \mathbf{x} + \alpha\\
    s.t \quad & (\alpha\mathbf{e}_1^T - \mathbf{d}^T\mathbf{Y})^T \in \mathcal{K}^*,\\
    &(\mathbf{A}\mathbf{x}\mathbf{e}_1^T + \mathbf{B}\mathbf{Y} - \mathbf{C})^T \in \mathcal{K}^*,\\
    &\mathbf{Y} \in \mathcal{K}^*, \mathbf{x} \in \mathcal{X}.
\end{aligned}
\end{gather}
We refer to this LP as the affine restriction of \ref{eq:ar}. We compare the affine policies given by our LP restriction \ref{eq:lp-ar} to the optimal affine policy in terms of worst-case objective value and running time.

\vspace{2mm} \noindent \textbf{Algorithm of El Housni and Goyal \cite{housni2021affine}.} 
 El Housni and Goyal \cite{housni2021affine} give an algorithm to efficiently compute near-optimal affine policies that show that achieve a $O(\frac{\log n}{\log \log n})$ approximation guarantee for the case of a single budget of uncertainty set. In particular, they consider affine policies of the following form. Let $$\mathbf{v_i} \in \argmin_{\mathbf{y} \geq 0}\{\mathbf{d}^T\mathbf{y} \mid \mathbf{B} \mathbf{y} \geq \mathbf{e}_i\},\; \forall i \in [m].$$
 El Housni and Goyal \cite{housni2021affine} consider affine policies of the form:
 $$\mathbf{y}(\mathbf{h})=\sum_i \nu_i \mathbf{v}_i h_i + \mathbf{q},$$
 for single budget of uncertainty. These can be computed using a linear program. 

 For our comparison benchmark, we consider a generalization of the algorithm in El Housni and Goyal \cite{housni2021affine} that computes the optimal affine policy of the form $\mathbf{y}(\mathbf{h})=\sum_i \nu_i \mathbf{v}_i h_i + \mathbf{q}$ for the case of packing uncertainty set given by $\mathcal{U}=\{\mathbf{h} \geq \mathbf{0} \;|\; \mathbf{R}\mathbf{h} \leq \mathbf{r}\}$ (we refer the reader to Appendix \ref{appendix:proof_approx_affine} for an expression of the corresponding LP as well as its derivation). We denote the corresponding LP by \ref{eq:eg}. We compare the affine policies given by our restriction \ref{eq:lp-ar} to those given by \ref{eq:eg} in terms of objective value and running time.

\vspace{2mm} \noindent \textbf{Lower-Bound of Hadjiyiannis et al. \cite{hadjiyiannis2011scenario}.} 
 Hadjiyiannis et al. \cite{hadjiyiannis2011scenario} propose a scenario based lower-bound for~\ref{eq:ar}. In particular, they consider  a finite set of scenarios $\Delta \subset \mathcal{U}$ referred to as a critical set, and solve a relaxation of the two-stage problem for these scenarios. We denote such relaxation as $AR(\Delta)$. Hadjiyiannis et al. \cite{hadjiyiannis2011scenario} propose an efficient algorithm to compute one such critical set that relies on the conic dual of the affine restriction \ref{eq:affine-ar} given as:

\begin{gather}
\label{eq:dual-AFF}
\tag{D-AFF}
\begin{aligned}
    \max_{\bm{\lambda}, \bm{\Lambda}, \bm{\Theta}} \quad & \sum_{i=1}^m \Lambda_{i+1,i}\\
    s.t \quad & 1-\mathbf{e}_1^T\bm{\lambda} = 0\\
    & \bm{\Lambda} \mathbf{B} +\bm{\Theta} = \bm{\lambda} \mathbf{d}^T\\\
    & \mathbf{e}_1^T\bm{\Lambda}\mathbf{A} \leq \mathbf{c}^T\\
    & \bm{\lambda} \in \mathcal{K}\\
    & \bm{\Lambda}_j \in \mathcal{K}, \quad \bm{\Theta}_j \in \mathcal{K}, \quad \forall j.\\
\end{aligned}
\end{gather}
Here $\bm{\Lambda}_j$ and $\bm{\Theta}_j$ denote the $j$-th column of $\bm{\Lambda}$ and $\bm{\Theta}$ respectively where $\bm{\lambda}, \bm{\Lambda}$ and $\bm{\Theta}$ are the dual variables corresponding the first, second and third conic constraints in \ref{eq:affine-ar} respectively. In particular, given an optimal dual solution $\bm{\lambda}^*, \bm{\Lambda}^*, \bm{\Theta}^*$ of \ref{eq:dual-AFF}, Hadjiyiannis et al. \cite{hadjiyiannis2011scenario} show that the critical set given by $\Delta_{\sf AFF} = \{\lambda^*\} \cup \{ \Lambda^*_j \mid \Lambda^*_j \neq \mathbf{0}\} \cup \{ \Theta^*_j \mid \Theta^*_j \neq \mathbf{0}\}$ 
is such that under $\Delta_{\sf AFF}$ the relaxation $AR(\Delta_{\sf AFF})$ gives an empirically good lower-bound for~\ref{eq:ar}. We compare the worst-case cost of our affine policies to the lower bound $AR(\Delta_{\sf AFF})$.

\paragraph{Experimental Setup.}
Following Ben-Tal et al. \cite{bental2018piecewiseaffine}, we consider instances of \ref{eq:ar} where $n=m$, $\mathbf{c}=\mathbf{d}=\mathbf{e}$ and $\mathbf{A}=\mathbf{B}=\mathbf{I}_{m}+\mathbf{G}$, where $\mathbf{I}_{m}$ is the identity matrix and $\mathbf{G}$ is a random normalized Gaussian matrix. We consider the  case where $\mathcal{X} = \mathbb{R}_+^m$ and $\mathcal{U}$ is an intersection of $L$ budget of uncertainty sets of the form:
$$\mathcal{U} = \left\{\mathbf{h} \in [0, 1]^m \mid \bm{\omega}_l^T \mathbf{h} \leq 1 \; \forall l \in [L]\right\},$$ where the weight vectors $\bm{\omega}_l$ are normalized Gaussian vectors, i.e., $\omega_{l,i} = \frac{|G_{l,i}|}{\sqrt{\sum_{i} (G_{l,i})^2}}$ for  $\{G_{l,i}\}$ i.i.d. standard Gaussian variables. We use a dual-core laptop with 8GB of RAM and 1.8GHz processor for the experiments and present the results in Tables \ref{table:ar-l20}, \ref{table:ar-l50} and \ref{table:ar-l100}. Here $z_{\sf AFF}$ denotes the optimum of the affine restriction \ref{eq:affine-ar}, $z_{\sf LP-AR}$ the optimum of our restriction \ref{eq:lp-ar}, $z_{\sf EG}$ the optimal value of \ref{eq:eg} and $z_{\sf LB}$ the lower-bound given by $AR(\Delta_{\sf AFF})$ as defined above. In terms of running time, $T_{\sf AFF}$ denotes the running time in seconds of the affine restriction, $T_{\sf LP-AR}$ the running time in seconds of our restriction and $T_{\sf EG}$ the running time in seconds of \ref{eq:eg}.

\paragraph{Discussion.} Tables \ref{table:ar-l20}, \ref{table:ar-l50} and \ref{table:ar-l100} show that our restriction \ref{eq:lp-ar} is significantly faster than both the generalization of El Housni and Goyal \cite{housni2021affine} \ref{eq:eg} and the optimal affine policy especially for higher dimensions. For instance, for $n=n'=100$ and $L=20$, our restriction is more than $30000$ times faster to compute than the optimal affine policy. For $n=n'=1000$ and $L=20$, our restriction is more than $85$ times faster than \ref{eq:eg}. Furthermore, for all values of $n$ we consider, our restriction has nearly the same objective value as \ref{eq:eg} (the ratio is less than $10^{-4}$ in all instances we consider). Note that \ref{eq:eg} optimizes over a richer class of affine policies than the affine policies we construct using \ref{eq:lp-ar} (\ref{eq:eg} optimizes over all affine policies of the form $\mathbf{y}(\mathbf{h})=\sum_i \nu_i \mathbf{v}_i h_i + \mathbf{q}$, while the affine policies we construct using \ref{eq:lp-ar} are such that there exists an $\bm{\alpha} \in \mathbb{R}_+^{L}$ such that $\nu_i = \frac{\mathbf{R}^T_i \bm{\alpha}}{\theta_i}$ for all $i \in [m]$) and it is a priori not clear that the two linear programs are equivalent. However, given the numerical findings, we conjecture that the optimal affine policies given by our restriction \ref{eq:lp-pdb} are optimal among the set of affine policies of the form of the form $\mathbf{y}(\mathbf{h})=\sum_i \nu_i \mathbf{v}_i h_i + \mathbf{q}$. Compared to the optimal affine policy, our restriction gives a good approximation that scales well with the dimension. For example, for $L=20$, our affine policy is within at most $20 \%$ for all $n=n' \leq 100$. Finally, compared to the optimum of \ref{eq:ar}, note that affine policies are less than $31\%$ from the Hadjiyiannis et al. \cite{hadjiyiannis2011scenario} lower-bound for all $n=n' \leq 100$ and $L=20$ implying that our restriction is within less than $69\%$ from the Hadjiyiannis et al. \cite{hadjiyiannis2011scenario} lower-bound for all $n=n' \leq 100$ and $L=20$. However, this comparison is with respect to a lower-bound of the optimum and not with respect to the optimum itself and includes the gap between the lower-bound and the optimum as well. Our computational study demonstrates that our restriction gives a good approximation of the two-stage adjustable robust problem that is significantly faster to compute than state-of-the-art approximation methods of this problem.

\begin{table}
    \centering
    \caption{Comparison of the objective value and the running time in seconds between our restriction \ref{eq:lp-ar}, the generalization of El Housni and Goyal \cite{housni2021affine} \ref{eq:eg}, and the optimal affine policy for different values of $n=n'$ and $L=20$. Entries denoted by a star exceeded a time limit of three hours.}
    \centering
        \begin{tabular}{|c|c|c|c|c|c|c|}
     \hline
     &&&&&& \\
     $n=n'$ & $T_{\sf LP-AR}$ & $T_{\sf EG}$ & $T_{\sf AFF}$ & $\frac{z_{\sf LP-AR}}{z_{\sf EG}}$ & $\frac{z_{\sf LP-AR}}{z_{\sf AFF}}$ & $\frac{z_{\sf AFF}}{z_{\sf LB}}$ \\[2ex]
     \hline
     20  &  0.052 &  0.080 & 0.412 &  1.000 &  1.304 &  1.197\\
     \hline
     40  &  0.074 &  0.273 &  21.866 &  1.000 &  1.226 &  1.303\\
     \hline
     60  &  0.104 &  0.609 & 202.732 &  1.000 &  1.248 &  1.310\\
     \hline
     80  &  0.309 &  1.726 &  822.509 &  1.000 &  1.207 &  1.298\\
     \hline
     100  &  0.153 & 3.227 & 4627.206 &  1.000 &  1.196  &  1.299 \\
     \hline
     500  &  0.540 & 22.064  & * &  1.000 &  *  &  *\\
     \hline
     1000  &  1.607 &  137.646 & * &  1.000 &  *  &  *\\
     \hline
    \end{tabular}
    \label{table:ar-l20}
\end{table}

\begin{table}
    \centering
    \caption{Comparison of the objective value and the running time in seconds between our restriction \ref{eq:lp-ar}, the generalization of El Housni and Goyal \cite{housni2021affine} \ref{eq:eg}, and the optimal affine policy for different values of $n=n'$ and $L=50$. Entries denoted by a star exceeded a time limit of three hours.}
    \centering
        \begin{tabular}{|c|c|c|c|c|c|c|}
     \hline
     &&&&&& \\
     $n=n'$ & $T_{\sf LP-AR}$ & $T_{\sf EG}$ & $T_{\sf AFF}$ & $\frac{z_{\sf LP-AR}}{z_{\sf EG}}$ & $\frac{z_{\sf LP-AR}}{z_{\sf AFF}}$ & $\frac{z_{\sf AFF}}{z_{\sf LB}}$ \\[2ex]
     \hline
     20 & 0.038 & 0.072 & 0.507 & 1.000 & 1.389 & 1.347\\
     \hline
     40 & 0.055 & 0.444 & 17.815 & 1.000 & 1.263 & 1.436\\
     \hline
     60 & 0.095 & 1.171 & 193.404 & 1.000 & 1.227 & 1.461\\
     \hline
     80 & 0.087 & 1.426 & 978.970 & 1.000 & 1.212 & 1.472\\
     \hline
     100 & 0.068 & 2.212 & 5649.392 & 1.000 & 1.190 & 1.420\\
     \hline
     500 & 0.580 & 104.684 & * & 1.000 & * & *\\
     \hline
     1000 & 1.756 & 2062.449 & * & 1.000 & * & * \\
     \hline
    \end{tabular}
    \label{table:ar-l50}
\end{table}

\begin{table}
    \centering
    \caption{Comparison of the objective value and the running time in seconds between our restriction \ref{eq:lp-ar}, the generalization of El Housni and Goyal \cite{housni2021affine} \ref{eq:eg}, and the optimal affine policy for different values of $n=n'$ and $L=100$. Entries denoted by a star exceeded a time limit of three hours.}
    \centering
        \begin{tabular}{|c|c|c|c|c|c|c|}
     \hline
     &&&&&& \\
     $n=n'$ & $T_{\sf LP-AR}$ & $T_{\sf EG}$ & $T_{\sf AFF}$ & $\frac{z_{\sf LP-AR}}{z_{\sf EG}}$ & $\frac{z_{\sf LP-AR}}{z_{\sf AFF}}$ & $\frac{z_{\sf AFF}}{z_{\sf LB}}$ \\[2ex]
     \hline
     20 & 0.073 & 0.179 & 1.166 & 1.000 & 1.369 & 1.381\\
     \hline
     40 & 0.046 & 0.605 & 18.804 & 1.000 & 1.285 & 1.521\\
     \hline
     60 & 0.119 & 1.894 & 188.047 & 1.000 & 1.230 & 1.588\\
     \hline
     80 & 0.138 & 2.230 & 1112.060 & 1.000 & 1.200 & 1.561 \\
     \hline
     100 & 0.122 & 5.606 & 2812.385 & 1.000 & 1.186 & 1.634\\
     \hline
     500 & 0.997 & 4986.632 & * & 1.000 & * & *\\
     \hline
     1000 & 1.189 & 9734.252 & * & 1.000 & * & *\\
     \hline
    \end{tabular}
    \label{table:ar-l100}
\end{table}

\section{Conclusion}
In this paper, we consider the class of packing disjoint bilinear programs \ref{eq:pdb} and present an LP rounding based randomized approximation algorithm for this problem. In particular, we show the existence of a near-optimal near-integral solution for \ref{eq:pdb}. We give an LP relaxation from which we obtain such solution using a randomized rounding of an optimal solution. We show that out relaxation is closely related to the reformulation linearization technique (RLT). We apply our ideas to the two-stage adjustable problem \ref{eq:ar} whose separation problem is a variant of \ref{eq:pdb}. While a direct application of the approximation algorithm for \ref{eq:pdb} does not work for \ref{eq:ar}, we derive an LP restriction of \ref{eq:ar}, based on similar ideas, that gives a polylogarithmic approximation of \ref{eq:ar}. We relate our LP restriction to affine policies. In particular, using an optimal solution of the LP restriction, we construct a near-optimal affine policy whose objective value is smaller than the optimal cost of the LP restriction. This proves that affine policies give a polylogarithmic approximation of \ref{eq:ar} and gives a new algorithm to compute near-optimal affine policies. We evaluate the numerical performance of our algorithms for \ref{eq:pdb} and \ref{eq:ar} and show that they are significantly faster and provide good empirical solutions.

\printbibliography
\newpage

\begin{appendices}

\section{Chernoff bounds}
\label{appendix:chernoff-bounds}
\noindent \textbf{Proof of Chernoff bounds $\hyperref[lm:chernoff-bounds]{(a)}$.} From Markov's inequality we have for all $t>0$,
 $$ \mathbb{P}( \Xi \geq (1+ \delta ) s ) = \mathbb{P}( e^{t\Xi} \geq e^{t(1+\delta)s} ) \leq \frac{\mathbb{E}(e^{t\Xi})}{e^{t(1+\delta)s}}.$$
 Denote $p_i$ the parameter of the Bernoulli ${\chi}_i$. By independence, we have
 $$  \mathbb{E}(e^{t\Xi})= \prod_{i=1}^r \mathbb{E} (e^{t \epsilon_i {\chi}_i}) = \prod_{i=1}^r \left( p_i  e^{t \epsilon_i }+1-p_i \right)
 \leq \prod_{i=1}^r \exp \left( p_i  (e^{t \epsilon_i }-1) \right)$$
 where the inequality holds because $1+x \leq e^x$ for all $x \in \mathbb{R}$.
 By taking $t=\ln( 1+ \delta) >0$, the right hand side becomes
 \begin{align*}
 \prod_{i=1}^r \exp \left( p_i  ((1+\delta)^{ \epsilon_i }-1) \right) \leq \prod_{i=1}^r \exp \left( p_i  \delta \epsilon_i \right) =\exp \left(   \delta  \cdot \mathbb{E}(\Xi ) \right) \leq e^{\delta s} ,\\
 \end{align*}
where the first inequality holds because $ (1+x)^{\epsilon} \leq 1+ \alpha x$ for any $x \geq 0$ and $\epsilon \in [0,1]$ and the second one because $s \geq \mathbb{E}(\Xi )= \sum_{i=1}^r \epsilon_i p_i$. Hence, 
we have  $$  \mathbb{E}(e^{t\Xi}) \leq e^{\delta s}. $$
On the other hand, $$e^{t(1+\delta)s} = (1+\delta) ^{(1+\delta)s}.$$
Therefore,
 $$ \mathbb{P}( \Xi \geq (1+ \delta ) s ) \leq  \left(  \frac{e^{\delta s}}{(1+\delta) ^{(1+\delta)s}} \right) =\left(  \frac{e^{\delta}}{(1+\delta)^{1+\delta} }    \right)^s. $$ \qed

\noindent \textbf{Proof of Chernoff bounds $\hyperref[lm:chernoff-bounds]{(b)}$.} From Markov's inequality we have for all $t< 0$,
 $$ \mathbb{P}( \Xi \leq (1- \delta ) \mathbb{E}(\Xi) ) = \mathbb{P}( e^{t\Xi} \geq e^{t(1-\delta)\mathbb{E}(\Xi)} ) \leq \frac{\mathbb{E}(e^{t\Xi})}{e^{t(1-\delta) \mathbb{E}(\Xi)}}.$$
 Denote $p_i$ the parameter of the Bernoulli ${\chi}_i$. By independence, we have
 $$  \mathbb{E}(e^{t\Xi})=    \prod_{i=1}^r \mathbb{E} (e^{t \epsilon_i {\chi}_i}) = \prod_{i=1}^r \left( p_i  e^{t \epsilon_i }+1-p_i \right)
 \leq  \prod_{i=1}^r \exp \left( p_i  (e^{t \epsilon_i }-1) \right) ,$$
 where the inequality holds because $1+x \leq e^x$ for all $x \in \mathbb{R}$.
We take $t=\ln( 1- \delta) < 0$. We have $ t \leq - \delta$, hence,
 \begin{align*}
\prod_{i=1}^r \exp \left( p_i  (e^{ t\epsilon_i }-1) \right)   =\prod_{i=1}^r \exp \left( p_i  ((1-\delta)^{ \epsilon_i }-1) \right) 
\leq \prod_{i=1}^r \exp \left( -p_i  \delta \epsilon_i \right),
 \end{align*}
where the  inequality holds because $ (1-x)^{\epsilon} \leq 1- \epsilon x$ for any $ 0 < x <1$ and $\epsilon \in [0,1]$. Therefore,
 $$  \mathbb{E}(e^{t\Xi}) \leq       \prod_{i=1}^r \exp \left( -p_i  \delta \epsilon_i \right)   =          e^{-\delta \mathbb{E}(\Xi)}. $$
On the other hand, $$e^{t(1-\delta) \mathbb{E}(\Xi)} = (1-\delta) ^{(1-\delta)\mathbb{E}(\Xi)}.$$
Therefore,
 $$ \mathbb{P}( \Xi \leq (1- \delta ) \mathbb{E}(\Xi) ) \leq  \left(  \frac{e^{-\delta \mathbb{E}(\Xi)}}{(1-\delta) ^{(1-\delta)\mathbb{E}(\Xi)}} \right) =\left(  \frac{e^{-\delta}}{(1-\delta)^{1-\delta} }    \right)^{\mathbb{E}(\Xi)}. $$
Finally, we have for any $0<\delta < 1$,
$$ \ln (1- \delta ) \geq  - \delta + \frac{\delta^2}{2}$$
which implies 
$$ (1- \delta) \cdot \ln (1- \delta ) \geq  - \delta + \frac{\delta^2}{2}$$
and consequently
$$\left(  \frac{e^{-\delta}}{(1-\delta)^{1-\delta} }    \right)^{\mathbb{E}(\Xi)} \leq e^{  \frac{-{\delta}^2 \mathbb{E}(\Xi)}{2}}. $$ \qed

\section{Proof of Theorem~\ref{th:inapproximability-cdb}}
\label{appendix:proof-inapproximability}
Our proof uses a polynomial time transformation from the \textit{Monotone Not-All-Equal 3-satisfiability} (MNAE3SAT) NP-complete problem (Schaefer \cite{Schaefer78}). In the (MNAE3SAT), we are given a collection of Boolean variables and a collection of clauses, each of which combines three variables. (MNAE3SAT) is the problem of determining if there exists a truth assignment where each close has at least one true and one false literal. This is a subclass of the \textit{Not-All-Equal 3-satisfiability} problem where the variables are never negated.

Consider an instance $\mathcal{I}$ of (MNAE3SAT), let $\mathcal{V}$ be the set of variables of $\mathcal{I}$ and let $\mathcal{C}$ be the set of clauses. Let $\mathbf{A} \in \{0,1\}^{|\mathcal{C}| \times |\mathcal{V}|}$ be the variable-clause incidence matrix such that for every variable $v \in \mathcal{V} $ and clause $c \in \mathcal{C}$ we have $A_{cv}=1$ if and only if the variable $v$ belongs to the clause $c$. We consider the following instance of \ref{eq:cdb} denoted by \ref{eq:cdbi},
\begin{gather}
\label{eq:cdbi}
\tag{$\mathcal{I}'$}
\begin{aligned}
\min_{\mathbf{x}, \mathbf{y}} & \quad \sum_{v \in \mathcal{V}} x_v y_v \\
s.t & \quad \mathbf{A}\mathbf{x} \geq \mathbf{e}, \quad \mathbf{x} \geq \mathbf{0} \\
& \quad \mathbf{A}\mathbf{y} \geq \mathbf{e}, \quad \mathbf{y} \geq \mathbf{0}.
\end{aligned}
\end{gather}

Let $z_{\mathcal{I}'}$ denote the optimum of \ref{eq:cdbi}. We show that $\mathcal{I}$ has a truth assignment where each clause has at least one true and one false literal if and only if $z_{\mathcal{I}'} = 0$.

First, suppose $z_{\mathcal{I}'} = 0$. Let $(\mathbf{x}, \mathbf{y})$ be an optimal solution of \ref{eq:cdbi}. Let $\Tilde{\mathbf{x}}$ such that $\Tilde{x}_v = \mathds{1}_{\{x_v > 0\}}$ for all $v \in \mathcal{V}$. 
We claim that $(\Tilde{\mathbf{x}}, \mathbf{e} - \Tilde{\mathbf{x}})$ is also an optimal solution of $\mathcal{I}$. In fact, for all $c \in \mathcal{C}$, we have, $$\sum_{v \in \mathcal{V}} A_{cv} x_v \geq 1,$$ hence, $$\sum_{v \in \mathcal{V} } A_{cv} \min(x_v, 1) \geq 1,$$ since the entries of $\mathbf{A}$ are in $\{0,1\}$. Therefore, $$\sum_{v \in \mathcal{V} } A_{cv} \Tilde{x}_v \geq \sum_{v \in \mathcal{V} } A_{cv} \min(x_v, 1) \geq 1.$$ Finally, $\mathbf{A} \Tilde{\mathbf{x}} \geq \mathbf{e}$. Similarly, for all $c \in \mathcal{C}$ we have, $$\sum_{v \in \mathcal{V} } A_{cv} y_v \geq 1,$$ hence, $$\sum_{v \in \mathcal{V} } A_{cv} \min(y_v, 1) \geq 1,$$ since the entries of $\mathbf{A}$ are in $\{0,1\}$. Therefore, $$\sum_{v \in \mathcal{V} } A_{cv} (1- \Tilde{x}_v) \geq \sum_{v \in \mathcal{V} } A_{cv} (1- \Tilde{x}_v) \min(y_v, 1) = \sum_{v \in \mathcal{V} } A_{cv} \min(y_v, 1) \geq 1.$$ where the equality follows from the fact that $y_v = 0$ for all $v$ such that $\Tilde{x}_v = 1$. Note that $\sum_{v \in \mathcal{V} } x_vy_v = z_{\mathcal{I}'} = 0$. This implies that $\mathbf{A}(\mathbf{e} - \Tilde{\mathbf{x}}) \geq \mathbf{e}$. Therefore, $(\Tilde{\mathbf{x}}, \mathbf{e} - \Tilde{\mathbf{x}})$ is a feasible solution of \ref{eq:cdbi} with objective value $\sum_{v \in \mathcal{V} } \Tilde{x}_v(1-\Tilde{x}_v) = 0 = z_{\mathcal{I}'}$. It is therefore an optimal solution. Consider now the truth assignment where a variable $v$ is set to be true if and only if $\Tilde{x}_v = 1$. We show that each clause in such assignment has at least one true and one false literal. In particular, consider a clause $c \in \mathcal{C}$, since, $$\sum_{v \in \mathcal{V} } A_{cv} \Tilde{x}_v \geq 1,$$ there must be a variable $v$ belonging to the clause $c$ such that $\Tilde{x}_v = 1$. Similarly, since $$\sum_{v \in \mathcal{V}} A_{cv} (1-\Tilde{x}_v) \geq 1,$$ there must be a variable $v$ belonging to the clause $c$ such that $\Tilde{x}_v = 0$. This implies that our assignment is such that the clause $c$ has at least one true and one false literal.

Conversely, suppose there exists a truth assignment of ${\mathcal I}$ where each clause has at least one false and one true literal. We show that $z_{\mathcal{I}'} = 0$. In particular, define $\mathbf{x}$ such that for all $v \in \mathcal{V}$ we have $x_v = 1$ if and only if $v$ is assigned true. We have for all $c \in \mathcal{C}$, $$\sum_{v \in \mathcal{V} } A_{cv} x_v \geq 1,$$ since at least one of the variables of $c$ is assigned true. And, $$\sum_{v \in \mathcal{V} } A_{cv} (1-x_v) \geq 1,$$ since at least one of the variables of $c$ is assigned false. Hence, $(\mathbf{x}, \mathbf{e} - \mathbf{x})$ is feasible for \ref{eq:cdbi} with objective value $\sum_{v \in \mathcal{V}} x_v(1-x_v) = 0$. It is therefore an optimal solution and $z_{\mathcal{I}'} = 0$.

If there exists a polynomial time algorithm approximating \ref{eq:cdb} to some finite factor, such an algorithm can be used to decide in polynomial time whether $z_{\mathcal{I}'} = 0$ or not, which is equivalent to solving $\mathcal{I}$ in polynomial time; a contradiction. \qed

\section{Derivation of the first level relaxation of the RLT hierarchy.}
\label{appendix:1RLT}

Let us begin by writing \ref{eq:pdb} in the following equivalent epigraph form,

\begin{gather*}
\begin{aligned}
\max_{\mathbf{x}, \mathbf{y}, \mathbf{u}, \mathbf{v}, \mathbf{w}} \quad & \sum_{i=1}^n \theta_i \gamma_i u_{ii}\\
& u_{ij} \leq x_i y_j, \quad \forall i, j\\
& v_{i,j} \leq x_i x_j, \quad \forall i, j\\
& w_{i,j} \leq y_i y_j, \quad \forall i, j\\
& \sum_{i=1}^n \theta_i\mathbf{P}_ix_i \leq \mathbf{p},\\
& \sum_{i=1}^n \gamma_i\mathbf{Q}_iy_i \leq \mathbf{q},\\
& 0 \leq x_i \leq 1, \quad 0 \leq y_i \leq 1, \quad \forall i.\\
\end{aligned}
\end{gather*}

\vspace{2mm}
\noindent \textbf{Reformulation phase.}
In the reformulation phase we add to the above program the polynomial constraints we get from multiplying all the linear constraints by linear constraints by $x_i$, $y_i$, $(1-x_i)$ and $(1-y_i)$ for all $i \in [n]$. We get the following equivalent formulation of \ref{eq:pdb},


\begin{gather*}
\begin{aligned}
\max_{\mathbf{x}, \mathbf{y}, \mathbf{u}, \mathbf{v}, \mathbf{w}} \quad & \sum_{i=1}^n \theta_i \gamma_j u_{ii}\\
& u_{ij} \leq x_i y_j, \quad \forall i, j,\\
& v_{i,j} \leq x_i x_j, \quad \forall i, j\\
& w_{i,j} \leq y_i y_j, \quad \forall i, j\\
& \sum_{j=1}^n \theta_j\mathbf{P}_jx_j \leq \mathbf{p},\\
& \sum_{j=1}^n \theta_j\mathbf{P}_{j}x_jx_i \leq \mathbf{p}x_i,\\
&\sum_{j=1}^n \theta_j\mathbf{P}_{j}(x_j- x_jx_i) \leq \mathbf{p} (1- x_i)\\
& \sum_{j=1}^n \theta_j\mathbf{P}_{j}x_jy_i \leq \mathbf{p}y_i,\\
&\sum_{j=1}^n \theta_j\mathbf{P}_{j}(x_j- y_ix_j) \leq \mathbf{p} (1- y_i)\\
& \sum_{j=1}^n \gamma_j\mathbf{Q}_jy_j \leq \mathbf{q},\\
& \sum_{j=1}^n \gamma_j\mathbf{Q}_{j}y_jx_i \leq \mathbf{q}x_i,\\
&\sum_{j=1}^n \gamma_j\mathbf{Q}_{j}(y_j- y_jx_i) \leq \mathbf{p} (1- x_i)\\
& \sum_{j=1}^n \gamma_j\mathbf{Q}_{j}y_jy_i \leq \mathbf{q}y_i,\\
&\sum_{j=1}^n \gamma_j\mathbf{Q}_{j}(y_j- y_jy_i) \leq \mathbf{p} (1- y_i)\\
& 0 \leq x_i \leq 1, \quad 0 \leq y_i \leq 1, \quad \forall i\\
& x_ix_j \leq x_i, \quad x_ix_j \leq x_j,\\ 
& y_iy_j \leq y_i, \quad y_iy_j \leq y_j,\\
&x_iy_j \leq x_i, \quad x_iy_j \leq y_j,\\
\end{aligned}
\end{gather*}

\vspace{2mm}
\noindent \textbf{Linearization phase.}
We replace the bilinear terms $x_iy_j$, $x_ix_j$ and $y_iy_j$ in the above LP with their respective lower-bounds $u_{ij}$, $v_{i,j}$, $w_{i,j}$. We get the following linear relaxation of \ref{eq:pdb},


\begin{gather*}
\begin{aligned}
\max_{\mathbf{x}, \mathbf{y}, \mathbf{u}, \mathbf{v}, \mathbf{w}} \quad & \sum_{i=1}^n \theta_i \gamma_j u_{ii}\\
& \sum_{j=1}^n \theta_j\mathbf{P}_jx_j \leq \mathbf{p},\\
& \sum_{j=1}^n \theta_j\mathbf{P}_{j}v_{i,j} \leq \mathbf{p}x_i,\\
&\sum_{j=1}^n \theta_j\mathbf{P}_{j}(x_j- v_{i,j}) \leq \mathbf{p} (1- x_i)\\
& \sum_{j=1}^n \theta_j\mathbf{P}_{j}u_{ji} \leq \mathbf{p}y_i,\\
&\sum_{j=1}^n \theta_j\mathbf{P}_{j}(x_j- u_{ji}) \leq \mathbf{p} (1- y_i)\\
& \sum_{j=1}^n \gamma_j\mathbf{Q}_jy_j \leq \mathbf{q},\\
& \sum_{j=1}^n \gamma_j\mathbf{Q}_{j}u_{ij} \leq \mathbf{q}x_i,\\
&\sum_{j=1}^n \gamma_j\mathbf{Q}_{j}(y_j- u_{ij}) \leq \mathbf{q} (1- x_i)\\
& \sum_{j=1}^n \gamma_j\mathbf{Q}_{j}w_{i,j} \leq \mathbf{q}y_i,\\
&\sum_{j=1}^n \gamma_j\mathbf{Q}_{j}(y_j- w_{i,j}) \leq \mathbf{q} (1- y_i)\\
& 0 \leq x_i \leq 1, \quad 0 \leq y_i \leq 1, \quad \forall i\\
& v_{i,j} \leq x_i, \quad v_{i,j} \leq x_j,\\ 
& w_{i,j} \leq y_i, \quad w_{i,j} \leq y_j,\\
&u_{ij} \leq x_i, \quad u_{ij} \leq y_j.\\
\end{aligned}
\end{gather*}
The above LP is known as the first level relaxation of the RLT hierarchy. The above LP can be further simplified as follows:  the constraints in $\mathbf{v}$ and $\mathbf{w}$ in the above are redundant and can be removed without loss of generality. In fact, for any feasible solution $\mathbf{x}, \mathbf{y}, \mathbf{u}$ of the resulting LP, one can take $v_{i,j}=x_ix_j$ and $w_{i,j}=y_iy_j$ to get a feasible solution of the above LP with same cost, and conversely for every feasible solution $\mathbf{x}, \mathbf{y}, \mathbf{u}, \mathbf{v}, \mathbf{w}$ of the above LP, the solution $\mathbf{x}, \mathbf{y}, \mathbf{u}$ is a feasible solution of the resulting LP with same cost. Also, the first (resp. sixth) set of constraints in the above LP can be obtained by summing the fourth and fifth (resp. seventh and eighth) set of constraints and can therefore be removed as well. We get the following equivalent LP relaxation of \ref{eq:pdb} that we refer to in this paper as the first level relaxation of the RLT hierarchy,

\begin{gather}
\label{eq:1rlt}
\tag{FL-RLT}
\begin{aligned}
\max_{\mathbf{x}, \mathbf{y}, \mathbf{u}} \quad & \sum_{i=1}^n \theta_i \gamma_j u_{ii}\\
& \sum_{j=1}^n \theta_j\mathbf{P}_{j}u_{ji} \leq \mathbf{p}y_i,\\
&\sum_{j=1}^n \theta_i\mathbf{P}_{j}(x_j- u_{ji}) \leq \mathbf{p} (1- y_i)\\
& \sum_{j=1}^n \gamma_j\mathbf{Q}_{j}u_{ij} \leq \mathbf{q}x_i,\\
&\sum_{j=1}^n \gamma_j\mathbf{Q}_{j}(y_j- u_{ij}) \leq \mathbf{q} (1- x_i)\\
&0 \leq u_{ij} \leq x_i \leq 1, \quad 0 \leq u_{ij} \leq y_j \leq 1.\\
\end{aligned}
\end{gather}

\section{Proof of Lemma~\ref{lm:structural-property-ar}.}
\label{appendix:proof_structural_lemma}

Let $\bm{\omega}^*$ be an optimal solution of $\hyperref[eq:qlp]{\mathcal{Q}^{\sf{LP}}(\beta\mathbf{x})}$, consider $\Tilde{\omega}_1, \dots, \Tilde{\omega}_m$ i.i.d. Bernoulli random variables such that $\mathbb{P}(\Tilde{\omega}_i = 1) = \omega^*_i$ for all $i \in [m]$ and let $(\mathbf{h}, \mathbf{z})$ such that $h_i = \frac{\gamma_i\Tilde{\omega}_i}{\beta}$ and $z_i = \frac{\theta_i\Tilde{\omega}_i}{\eta}$ for all $i \in [m]$. We show that $(\mathbf{h}, \mathbf{z})$ satisfies the properties~\eqref{eq:properties-ar} with a constant probability. In particular,
we have, 

\begin{gather}
\tag*{ }
\begin{aligned}
    \mathbb{P}(\sum_{i=1}^m \mathbf{b}_i z_i \nleq \mathbf{d}) 
    &= \mathbb{P}\left(\sum_{i=1}^m \theta_i\mathbf{b}_i\frac{\Tilde{\omega}_i}{\eta} \nleq \mathbf{d}\right) \\
    & \leq \sum_{j=1}^{n} \mathbb{P}\left(\sum_{i=1}^m \theta_i B_{ij}\frac{\Tilde{\omega}_i}{\eta} > d_j\right) \\
    &= \sum_{j \in [n]: d_j > 0} \mathbb{P}\left(\sum_{i=1}^m \frac{\theta_i B_{ij}}{d_j}\Tilde{\omega}_i > \eta\right) \\
    &\leq \sum_{j \in [n]: d_j > 0} \left(\frac{e^{\eta-1}}{\eta^{\eta}}\right)\\
    &\leq n \frac{e^{\eta-1}}{\eta^{\eta}},
\end{aligned}
\end{gather}
where the first inequality  follows from a union bound on n constraints. The second equality holds because for all $j \in [n]$ such that $d_j = 0$ we have  $$\mathbb{P}\left(\sum_{i=1}^m \theta_i B_{ij}\frac{\Tilde{\omega}_i}{\eta} > d_j\right) = 0.$$ Note that $d_j = 0$ implies $$\sum_{i=1}^m \theta_i B_{ij}\frac{\omega^*_i}{\eta} = 0,$$ by feasibility of $\bm{\omega}^*$ in $\hyperref[eq:qlp]{\mathcal{Q}^{\sf{LP}}(\beta\mathbf{x})}$. Therefore, $$\sum_{i=1}^m \theta_i B_{ij}\frac{\Tilde{\omega}_i}{\eta} = 0,$$ almost surely. The second inequality follows from the Chernoff bounds $\hyperref[lm:chernoff-bounds]{(a)}$ with $\delta = \eta - 1$ and $s=1$. In particular, $\frac{\theta_i B_{ij}}{d_j} \in [0,1]$ by definition of $\theta_i$ for all $i \in [m]$ and $j \in [n]$ such that $d_j > 0$ and for all $j \in [n]$ such that $d_j > 0$ we have, $$\mathbb{E}\left[\sum_{i=1}^m \frac{\theta_i B_{ij}}{d_j}\Tilde{\omega}_i\right] = \sum_{i=1}^m \frac{\theta_i B_{ij}}{d_j}\omega^*_i \leq 1,$$ by feasibility of $\bm{\omega}^*$. Next, note that $$\frac{e^{\eta-1}}{\eta^{\eta}} = O(\frac{1}{n^2}).$$ Therefore, there exists a constant $c > 0$ such that,
\begin{gather}
    \label{eq:property1-ar}
    \mathbb{P}(\sum_{i=1}^m \mathbf{b}_iz_i \nleq \mathbf{d})  \leq \frac{c}{n}.
\end{gather}
By a similar argument there exists a constant $c' > 0$,
\begin{gather}
\label{eq:property2-ar}
    \begin{aligned}
    \mathbb{P}(\sum_{i=1}^m \mathbf{R}_ih_i \nleq \mathbf{q}) \leq \frac{c'}{L},
\end{aligned}
\end{gather}
Finally we have,
\begin{gather}
\tag*{ }
\begin{aligned}
    &\mathbb{P}\left(\sum_{i=1}^m h_i z_i - (\mathbf{a}_i^T\mathbf{x}) z_i < \frac{1}{2\eta\beta}\mathcal{Q}^{\sf{LP}}(\beta\mathbf{x})\right) \\
    = &\mathbb{P}\left(\sum_{i=1}^m \frac{\theta_i\gamma_i\Tilde{\omega}^2_i}{\eta\beta} - \theta_i(\mathbf{a}_i^T\mathbf{x}) \frac{\Tilde{\omega}_i}{\eta} < \frac{1}{2\eta\beta}\mathcal{Q}^{\sf{LP}}(\beta\mathbf{x})\right) \\
    = &\mathbb{P}\left(\frac{1}{\eta\beta}\sum_{i=1}^m (\theta_i\gamma_i -  \theta_i\mathbf{a}_i^T(\beta\mathbf{x}))\Tilde{\omega}_i < \frac{1}{2\eta\beta}\mathcal{Q}^{\sf{LP}}(\beta\mathbf{x})\right).
\end{aligned}
\end{gather}
Let $\mathcal{I}$ denote the subset of indices $i \in [m]$ such that $$\theta_i\gamma_i -  \theta_i\mathbf{a}_i^T(\beta\mathbf{x}) \geq 0.$$ Since $\bm{\omega}^*$ is the optimal solution of the maximization problem $\hyperref[eq:qlp]{\mathcal{Q}^{\sf{LP}}(\beta\mathbf{x})}$ we can suppose without loss of generality that $\omega^*_i = 0$ for all $i \notin \mathcal{I}$. In fact, the packing constraints of $\hyperref[eq:qlp]{\mathcal{Q}^{\sf{LP}}(\beta\mathbf{x})}$ are down-closed (i.e., for every $\bm{\omega} \leq \bm{\omega}'$, if $\bm{\omega}'$ is feasible then $\bm{\omega}$ is also feasible), hence setting $\omega^*_i = 0$ for all $i \notin \mathcal{I}$ still gives a feasible solution of $\hyperref[eq:qlp]{\mathcal{Q}^{\sf{LP}}(\beta\mathbf{x})}$ and can only increase the objective value. Hence, $\Tilde{\omega}_i = 0$ almost surely for all $i \notin \mathcal{I}$ and we have,
\begin{gather}
\label{eq:property3-ar}
\begin{aligned}
    &\mathbb{P}\left(\sum_{i=1}^m h_i z_i - (\mathbf{a}_i^T\mathbf{x}) z_i < \frac{1}{2\eta\beta}\mathcal{Q}^{\sf{LP}}(\beta\mathbf{x})\right) \\
    = &\mathbb{P}\left(\frac{1}{\eta\beta}\sum_{i \in \mathcal{I}} (\theta_i\gamma_i -  \theta_i\mathbf{a}_i^T(\beta\mathbf{x}))\Tilde{\omega}_i < \frac{1}{2\eta\beta}\mathcal{Q}^{\sf{LP}}(\beta\mathbf{x})\right)\\
    = &\mathbb{P}\left(\sum_{i \in \mathcal{I}} \frac{(\theta_i\gamma_i -  \theta_i\mathbf{a}_i^T(\beta\mathbf{x}))}{\mathcal{Q}^{\sf{LP}}(\beta\mathbf{x})}\Tilde{\omega}_i < \frac{1}{2}\right) \leq e^{-\frac{1}{8}},
\end{aligned}
\end{gather}
where the last inequality follows from Chernoff bounds $\hyperref[lm:chernoff-bounds]{(b)}$ with $\delta=1/2$. In particular we have for all  $i \in \mathcal{I}$ $$\frac{(\theta_i\gamma_i -  \theta_i\mathbf{a}_i^T(\beta\mathbf{x}))}{\mathcal{Q}^{\sf{LP}}(\beta\mathbf{x})} \leq 1.$$ This is because the unit vector $\mathbf{e}_i$ is feasible for $\hyperref[eq:qlp]{\mathcal{Q}^{\sf{LP}}(\beta\mathbf{x})}$ for all $i \in \mathcal{I}$ {\color{blue} which implies} $$(\theta_i\gamma_i -  \theta_i\mathbf{a}_i^T(\beta\mathbf{x})) \leq \mathcal{Q}^{\sf{LP}}(\beta\mathbf{x}).$$ We also have, $$\mathbb{E}\left [\sum_{i \in \mathcal{I}} \frac{(\theta_i\gamma_i -  \theta_i\mathbf{a}_i^T(\beta\mathbf{x}))}{\mathcal{Q}^{\sf{LP}}(\beta\mathbf{x})}\Tilde{\omega}_i\right] = 1.$$ Combining inequalities (\ref{eq:property1-ar}), (\ref{eq:property2-ar}) and (\ref{eq:property3-ar}) we get that $(\mathbf{h}, \mathbf{z})$ verifies the properties~\eqref{eq:properties-ar} with probability at least $$1-\frac{c}{n}-\frac{c'}{L}-e^{-\frac{1}{8}},$$ which is greater than a constant for $n$ and $L$ large enough. Which concludes the proof of the structural property. \qed

\section{On Assumption \ref{assumption}.}
\label{appendix:generalA}

Assumption~\ref{assumption} can be made without loss of generality. To show this, we construct for every instance $I$ of \ref{eq:ar} a new instance $\Tilde{I}$ such that Assumption \ref{assumption} holds under $\Tilde{I}$ and the optimal value of $\Tilde{I}$ is within a factor $2$ of the value of $I$. This implies that our LP approximation 
from Section~\ref{section:ar} under $\Tilde{I}$ is an $O(\frac{\log n}{\log \log n} \frac{\log L}{\log \log L})$ approximation for  $I$. In particular, consider an instance $I$ of \ref{eq:ar} given by,
\begin{gather}
\tag*{ }
\begin{aligned}
z_{I} = \min_{\mathbf{x} \in \mathcal{X}} \quad \mathbf{c}^T\mathbf{x} + \max_{\mathbf{h} \in \mathcal{U}} \min_{\mathbf{y}\geq \mathbf{0}} \; \{\mathbf{d}^T\mathbf{y} \mid
\mathbf{A}\mathbf{x} + \mathbf{B}\mathbf{y}\geq \mathbf{h}\},
\end{aligned}
\end{gather}
To $I$ we associate the following modified instance $\Tilde{I}$,
\begin{gather}
\tag*{ }
\begin{aligned}
z_{\Tilde{I}} = \min_{(\mathbf{x}, \mathbf{y}_0) \in \Tilde{\mathcal{X}}} \quad (\mathbf{c}^T \; \mathbf{d}^T) \left(\begin{matrix}
\mathbf{x}\\\mathbf{y}_0
\end{matrix}\right) + \max_{\mathbf{h} \in \mathcal{U}} \min_{\mathbf{y}\geq \mathbf{0}} \; \left\{\mathbf{d}^T\mathbf{y}\; \left|\;
(\mathbf{A} \; \mathbf{B})
\left(\begin{matrix}
\mathbf{x}\\\mathbf{y}_0
\end{matrix}\right)
 + \mathbf{B}\mathbf{y}\geq \mathbf{h}\right.\right\},
\end{aligned}
\end{gather}
where $$\Tilde{\mathcal{X}} = \left\{(x, y_0) 
\;\left|\; \begin{matrix}
    \mathbf{x} \in \mathcal{X}\\
    \mathbf{y}_0 \geq 0\\
    \mathbf{A}\mathbf{x}+\mathbf{B}\mathbf{y}_0 \geq 0
\end{matrix}
\right.\right\}.$$
Note that $\Tilde{I}$ is indeed an instance of \ref{eq:ar} as the first and second-stage cost vectors $\left(\begin{matrix}
\mathbf{c}\\\mathbf{d}
\end{matrix}\right)$ and $\mathbf{d}$ and the second-stage matrix $\mathbf{B}$ all have non-negative coefficients and the first stage feasible set $\Tilde{\mathcal{X}}$ is still a polyhedral cone. Assumption \ref{assumption} is verified under $\Tilde{I}$ by definition of $\Tilde{\mathcal{X}}$. We now show that $\Tilde{I}$ gives a $2$-approximation of $I$. In particular, we prove the following lemma,

\begin{lemma}
    $z_{I} \leq z_{\Tilde{I}} \leq 2 z_{I}$
\end{lemma}
\begin{proof}
    First of all, let $\Tilde{\mathbf{x}}^*, \Tilde{\mathbf{y}}_0^*$ be an optimal solution of $\Tilde{I}$. For every $h \in \mathcal{U}$ we have,
    \begin{align*}
        \mathbf{d}^T \mathbf{\Tilde{y}}_0^* + \min_{\mathbf{y} \geq \mathbf{0}} \left\{\mathbf{d}^T \mathbf{y} \;|\;
        \mathbf{A}\mathbf{\Tilde{x}}^* + \mathbf{B}(\mathbf{y} + \mathbf{\Tilde{y}}_0^*) \geq \mathbf{h}\right\}
        &=
        \min_{\substack{\mathbf{y} \geq \mathbf{\Tilde{y}}_0^*}} \left\{\mathbf{d}^T \mathbf{y} \;|\;
        \mathbf{A}\mathbf{\Tilde{x}}^* + \mathbf{B}\mathbf{y} \geq \mathbf{h}\right\}\\
        &\geq \min_{\mathbf{y}\geq 0} \left\{\mathbf{d}^T \mathbf{y} \;|\;
        \mathbf{A}\mathbf{\Tilde{x}}^* + \mathbf{B}\mathbf{y} \geq \mathbf{h}\right\}.
    \end{align*}
Hence,

\begin{align*}
        z_{\Tilde{I}} &= \mathbf{c}^T\mathbf{\Tilde{x}}^* + \mathbf{d}^T \mathbf{\Tilde{y}}_0^* + \max_{h \in \mathcal{U}}\min_{\mathbf{y} \geq \mathbf{0}} \left\{\mathbf{d}^T \mathbf{y} \;|\;
        \mathbf{A}\mathbf{\Tilde{x}}^* + \mathbf{B}(\mathbf{y} + \mathbf{\Tilde{y}}_0^*) \geq \mathbf{h}\right\}\\
        &\geq 
        \mathbf{c}^T\mathbf{\Tilde{x}}^* + \max_{h \in \mathcal{U}}\min_{\mathbf{y}\geq 0} \left\{\mathbf{d}^T \mathbf{y} \;|\;
        \mathbf{A}\mathbf{\Tilde{x}}^* + \mathbf{B}\mathbf{y} \geq \mathbf{h}\right\}\\
        &\geq z_{I}.
\end{align*}
where the last inequality follows by feasibility of $\mathbf{\Tilde{x}}^*$ for $I$. For the inverse inequality, let $\mathbf{x}^*$ be an optimal solution of $I$, and let $\mathbf{y}(0) \in \argmin_{\mathbf{y}\geq 0} \left\{\mathbf{d}^T \mathbf{y} \;|\; \mathbf{A}\mathbf{x}^* + \mathbf{B}\mathbf{y} \geq \mathbf{0}\right\}$. We have,
\begin{align*}
     z_{I} 
     &= \mathbf{c}^T\mathbf{x}^* + \max_{h \in \mathcal{U}}\min_{\mathbf{y}\geq 0} \left\{\mathbf{d}^T \mathbf{y} \;|\;
        \mathbf{A}\mathbf{x}^* + \mathbf{B}\mathbf{y} \geq \mathbf{h}\right\}\\
    &\geq 
    \mathbf{c}^T\mathbf{x}^* + \frac{1}{2}\left(\mathbf{d}^T \mathbf{y}(0) + \max_{h \in \mathcal{U}}\min_{\mathbf{y}\geq 0} \left\{\mathbf{d}^T \mathbf{y} \;|\;
        \mathbf{A}\mathbf{x}^* + \mathbf{B}\mathbf{y} \geq \mathbf{h}\right\}\right)\\
    &\geq 
    \mathbf{c}^T\mathbf{x}^* + \frac{1}{2}\left(\mathbf{d}^T \mathbf{y}(0) + \max_{h \in \mathcal{U}}\min_{\mathbf{y}\geq 0} \left\{\mathbf{d}^T \mathbf{y} \;|\;
        \mathbf{A}\mathbf{x}^* + \mathbf{B}\mathbf{y} + \mathbf{B}\mathbf{y}(0) \geq \mathbf{h}\right\}\right)\\
    &\geq \frac{1}{2} z_{\Tilde{I}}
\end{align*}
where the first inequality follows from the fact that $\mathbf{0}$ is a feasible scenario of the uncertainty set and the last inequality follows by feasibility of $(\mathbf{x}^*, \mathbf{y}(0))$ for $\Tilde{I}$ and because $\mathbf{c}^T\mathbf{x}^* \geq 0$. \qed
\end{proof}

\begin{remark}
    Note that for every feasible affine policy of $\Tilde{I}$ given by the first stage solution $(\Tilde{\mathbf{x}}, \Tilde{\mathbf{y}}_0)$ and the second-stage affine function $\Tilde{\mathbf{y}}_{\sf AFF}$, the affine policy given by the first stage solution $ \Tilde{\mathbf{x}}$ and the second-stage affine function $\mathbf{y}_{\sf AFF} =\Tilde{\mathbf{y}}_{\sf AFF} + \Tilde{\mathbf{y}}_0$ is a feasible affine policy of $I$ with same worst-case cost. This implies that the results of Section~\ref{section:affine} also hold in the general case. In particular, the affine policies constructed in Section~\ref{section:affine} under $\Tilde{I}$ can be used to construct affine policies for $I$ that are an $O(\frac{\log n}{\log \log n} \frac{\log L}{\log \log L})$ approximation for $I$.
\end{remark}

An example of an application of the two-stage problem where Assumption~\ref{assumption} does not hold is the following two-stage network design problem: consider a directed graph $D(V,A)$ where each arc $a \in A$ is associated with a first-stage cost $c_a \geq 0$ and each node $v \in V$ is associated with a second stage cost $d_v \geq 0$ and receives an uncertain demand $h_v \geq 0$. The decision maker chooses a flow vector $f \in \mathbb{R}_+^{A}$ where $f_{vw}$ represents the quantity of supply to node $w$ coming from node $v$ and incurs the first stage cost $\sum_{vw} c_{vw} f_{vw}$, the demands are then revealed and the decision maker incurs a second-stage cost $d_v \cdot (h_v + \sum_{w: vw \in A} f_{vw} - \sum_{w: wv \in A} f_{wv})^+$ for each node $v$ where the flow balance $\sum_{w: wv \in A} f_{wv} - \sum_{w: vw \in A} f_{vw}$ is lower than the demand $h_v$. This example can be modeled as an instance of \ref{eq:ar} where $\mathbf{B}=\mathbf{I}$ and $\mathbf{A}$ is the incidence matrix of the considered directed graph. The flow vector $f$ such that $f_{vw}=1$ for some arc $vw$ and $f_{v'w'}=0$ for every other arc $v'w'$ is such that $(\mathbf{A}\mathbf{f})_v = -1 < 0$.

\section{Derivation of the LP formulation corresponding to the generalization of Algorithm 2 of El Housni and Goyal \cite{housni2021affine} to packing uncertainty sets.}
\label{appendix:proof_approx_affine}
When the second-stage variable $\mathbf{y}$ is restricted to affine policies of the form $\mathbf{y}(\mathbf{h})=\sum_i \nu_i \mathbf{v}_i h_i + \mathbf{q},$ for some $\nu_1, \dots, \nu_m \in \mathbb{R}$ and $\mathbf{q} \in \mathbb{R}^{n}$, the two-stage problem \ref{eq:ar} becomes,

\begin{align*}
\min \quad & \mathbf{c}^T\mathbf{x} + z\\
& z \geq \mathbf{d}^T(\sum_i \nu_i \mathbf{v}_i h_i + \mathbf{q}), \quad \forall \mathbf{h} \in \mathcal{U}\\
& \mathbf{A}\mathbf{x} + \mathbf{B}(\sum_i \nu_i \mathbf{v}_i h_i + \mathbf{q}) \geq \mathbf{h}, \quad \forall \mathbf{h} \in \mathcal{U}\\
&\sum_i \nu_i \mathbf{v}_i h_i + \mathbf{q} \geq \mathbf{0}, \quad \forall \mathbf{h} \in \mathcal{U}\\
&\mathbf{x}\in \mathcal{X}, \; \mathbf{q} \in \mathbb{R}^n, \; \nu_i \in \mathbb{R}, \; z \in \mathbb{R},
\end{align*}

We use standard duality techniques to derive formulation \ref{eq:eg}. The first constraint is equivalent to 
$$ z - \mathbf{d}^T \mathbf{q} \geq \underset {\mathbf h \geq \mathbf 0}{\underset{  \mathbf R  \mathbf h \leq \mathbf r  }\max } \; \sum_i \nu_i \mathbf{d}^T\mathbf{v}_i h_i .$$
By taking the dual of the maximization problem, the constraint is equivalent to
$$ z- \mathbf{d}^T \mathbf{q} \geq \underset {\mathbf{v} \geq \mathbf 0 }{\underset{\mathbf R^T \mathbf v \geq (\mathbf Y \cdot {\sf diag}  ( \nu_1, \dots, \nu_m  ))^T \mathbf d}\min } \;\mathbf r^T \mathbf v.$$
Where $\mathbf{Y} := [\mathbf{v}_1, \dots, \mathbf{v}_m]$. We then drop the min and introduce $\mathbf v $ as a variable to obtain the following linear constraints,
$$ z- \mathbf d^T \mathbf q \geq \mathbf r^T \mathbf v $$
$$ \mathbf R^T \mathbf v \geq (\mathbf Y \cdot {\sf diag}  ( \nu_1, \dots, \nu_m  ))^T \mathbf d $$
$$ \mathbf{v}   \in  {\mathbb R}^{L}_+.$$
We use the same technique for the second sets of constraints, i.e.,
$$ \mathbf{A}\mathbf{x} + \mathbf{B}   \mathbf q   \ \; \geq \;    \underset {\mathbf h \geq \mathbf 0 }{\underset{  \mathbf R  \mathbf h \leq \mathbf r  }\max } \; (  \mathbf I_m - \mathbf B \mathbf Y \cdot {\sf diag}  ( \nu_1, \dots, \nu_m) ) \mathbf{h}  .$$
By taking the dual of the maximization problem for each row and dropping the $\min$ we get the following formulation of these constraints
$$ \mathbf A \mathbf x + \mathbf B \mathbf q \geq \mathbf V^T \mathbf r $$
$$ \mathbf V^T \mathbf R \geq \mathbf I_m - \mathbf B \mathbf Y \cdot {\sf diag}  ( \nu_1, \dots, \nu_m ) $$
$$  \mathbf{V}   \in  {\mathbb R}^{L \times m}_+ .$$
Similarly, the last constraint
$$ \mathbf{q} \; \geq \;\underset {\mathbf h \geq \mathbf 0 }{\underset{  \mathbf R  \mathbf h \leq \mathbf r  }\max } \;      -\mathbf Y \cdot {\sf diag}  ( \nu_1, \dots, \nu_m ) \mathbf{h},  $$
is equivalent to 
$$ \mathbf q \geq \mathbf U^T \mathbf r $$
$$ \mathbf U^T \mathbf R + \mathbf Y \cdot {\sf diag}  ( \nu_1, \dots, \nu_m  ) \geq \mathbf 0 $$
$$\mathbf{U}   \in  {\mathbb R}^{L \times n}_+ . $$
Putting all together, we get the following formulation,

\begin{gather}
\label{eq:eg}
\tag{EG}
\begin{aligned}
z_{\sf EG}= \min \; &  \mathbf{c}^T \mathbf{x} + z \\
& z- \mathbf d^T \mathbf q \geq \mathbf r^T \mathbf v \\
& \mathbf R^T \mathbf v \geq (\mathbf Y \cdot {\sf diag}  ( \nu_1, \dots, \nu_m  ))^T \mathbf d \\
& \mathbf A \mathbf x + \mathbf B \mathbf q \geq \mathbf V^T \mathbf r\\
& \mathbf V^T \mathbf R \geq \mathbf I_m - \mathbf B \mathbf Y \cdot {\sf diag}  ( \nu_1, \dots, \nu_m ) \\
& \mathbf q \geq \mathbf U^T \mathbf r \\
& \mathbf U^T \mathbf R + \mathbf Y \cdot {\sf diag}  ( \nu_1, \dots, \nu_m  ) \geq \mathbf 0 \\
& {\mathbf{x}  \in  {\cal X} }, \; \mathbf{v}   \in  {\mathbb R}^{L}_+, \; \mathbf{U}   \in  {\mathbb R}^{L \times n}_+, \; \mathbf{V}   \in  {\mathbb R}^{L \times m}_+ \\
&  \mathbf q \in \mathbb{R}^n, \; \nu_1, \dots, \nu_m \in {\mathbb R}, \; z \in \mathbb{R},
\end{aligned}
\end{gather}
\qed

\end{appendices}

\end{document}